\documentclass[english]{article}

\usepackage[T1]{fontenc}
\usepackage[latin9]{inputenc}
\usepackage{geometry}
\usepackage{color}
\usepackage{babel}
\usepackage{refstyle}
\usepackage{float}
\usepackage{amsmath}
\usepackage{amsthm}
\usepackage{amssymb}
\usepackage[all]{xy}
\usepackage[unicode=true,
 bookmarks=true,bookmarksnumbered=false,bookmarksopen=false,
 breaklinks=false,pdfborder={0 0 1},backref=false,colorlinks=true]
 {hyperref}
\hypersetup{pdftitle={The Conservative Matrix Field},
 pdfauthor={Ofir David},
 pdfkeywords={generalized continued fraction, irrationality}}

\makeatletter

\AtBeginDocument{\providecommand\subsecref[1]{\ref{subsec:#1}}}
\AtBeginDocument{\providecommand\thmref[1]{\ref{thm:#1}}}
\AtBeginDocument{\providecommand\appref[1]{\ref{app:#1}}}
\AtBeginDocument{\providecommand\secref[1]{\ref{sec:#1}}}
\AtBeginDocument{\providecommand\eqref[1]{\ref{eq:#1}}}
\AtBeginDocument{\providecommand\corref[1]{\ref{cor:#1}}}
\AtBeginDocument{\providecommand\lemref[1]{\ref{lem:#1}}}
\AtBeginDocument{\providecommand\exaref[1]{\ref{exa:#1}}}
\AtBeginDocument{\providecommand\claimref[1]{\ref{claim:#1}}}
\providecommand{\tabularnewline}{\\}
\RS@ifundefined{subsecref}
  {\newref{subsec}{name = \RSsectxt}}
  {}
\RS@ifundefined{thmref}
  {\def\RSthmtxt{theorem~}\newref{thm}{name = \RSthmtxt}}
  {}
\RS@ifundefined{lemref}
  {\def\RSlemtxt{lemma~}\newref{lem}{name = \RSlemtxt}}
  {}

\theoremstyle{plain}
\newtheorem{thm}{\protect\theoremname}
\theoremstyle{remark}
\newtheorem{rem}[thm]{\protect\remarkname}
\theoremstyle{definition}
\newtheorem{defn}[thm]{\protect\definitionname}
\theoremstyle{plain}
\newtheorem{lem}[thm]{\protect\lemmaname}
\theoremstyle{plain}
\newtheorem{claim}[thm]{\protect\claimname}
\theoremstyle{plain}
\newtheorem{cor}[thm]{\protect\corollaryname}
\theoremstyle{definition}
\newtheorem{example}[thm]{\protect\examplename}
\theoremstyle{definition}
\newtheorem{problem}[thm]{\protect\problemname}

\@ifundefined{date}{}{\date{}}
\hypersetup{pdfstartview=}
\newref{def}{name=Definition~}
\newref{ex}{name=Example~}
\newref{eq}{name=Equation~}
\newref{exa}{name=Example~}
\newref{sub}{name=Subsection~}
\newref{cor}{name=Corollary~}
\newref{rem}{name=Remark~}
\newref{prop}{name=Proposition~}
\newref{lem}{name=Lemma~}
\newref{thm}{name=Theorem~}
\newref{sec}{name=Section~}
\newref{claim}{name=Claim~}
\newref{app}{name=Appendix~}

\hypersetup{citecolor=blue}

\newcommand{\xyR}[1]{
  \xydef@\xymatrixrowsep@{#1}}
\newcommand{\xyC}[1]{
  \xydef@\xymatrixcolsep@{#1}}

\usepackage{url}

\makeatother

\providecommand{\claimname}{Claim}
\providecommand{\corollaryname}{Corollary}
\providecommand{\definitionname}{Definition}
\providecommand{\examplename}{Example}
\providecommand{\lemmaname}{Lemma}
\providecommand{\problemname}{Problem}
\providecommand{\remarkname}{Remark}
\providecommand{\theoremname}{Theorem}

\begin{document}
\global\long\def\call{\mathcal{L}}%
\global\long\def\nn{\mathcal{N}}%
\global\long\def\ff{\mathcal{F}}%
\global\long\def\aa{\mathcal{A}}%
\global\long\def\RR{\mathbb{R}}%
\global\long\def\EE{\mathbb{E}}%
\global\long\def\CC{\mathbb{C}}%
\global\long\def\QQ{\mathbb{Q}}%
\global\long\def\ZZ{\mathbb{Z}}%
\global\long\def\NN{\mathbb{N}}%
\global\long\def\KK{\mathbb{K}}%
\global\long\def\SL{\mathrm{SL}}%
\global\long\def\GL{\mathrm{GL}}%
\global\long\def\ds{\mathrm{ds}}%
\global\long\def\dnu{\mathrm{d\nu}}%
\global\long\def\dmu{\mathrm{d\mu}}%
\global\long\def\dt{\mathrm{dt}}%
\global\long\def\dw{\mathrm{dw}}%
\global\long\def\dx{\mathrm{dx}}%
\global\long\def\dy{\mathrm{dy}}%
\global\long\def\norm#1{\left\Vert #1\right\Vert }%
\global\long\def\limfi#1{{\displaystyle \lim_{#1\to\infty}}}%
\global\long\def\arrfi#1{\overset{#1\to\infty}{\longrightarrow}}%
\global\long\def\flr#1{\left\lfloor #1\right\rfloor }%
\global\long\def\lcm{\mathrm{lcm}}%

\title{On Euler polynomial continued fractions}
\author{Ofir David\\
The Ramanujan Machine team, Faculty of Electrical and Computer Engineering,\\
Technion - Israel Institute of Technology, Haifa 3200003, Israel}
\maketitle
\begin{abstract}
In this paper, we introduce the polynomial continued fractions, a
close relative of the well-known simple continued fraction expansions
which are widely used in number theory and in general. While they
may not possess all the intriguing properties of simple continued
fractions, polynomial continued fractions have many interesting patterns
which can be exploited. Specifically, we explore the Euler continued
fractions within this framework and present an algorithm for their
identification.
\end{abstract}

\section{Introduction}

The simple continued fraction expansion is a central and important
mathematical tool used in many areas in mathematics and outside it.
There are many interesting properties that can be extracted from these
expansions, and in particular properties which show how ``close''
a number is to being rational. For example, some interesting expansions
for the golden ratio $\varphi$, $\sqrt{2}$ and $e$ are
\begin{align*}
\varphi & :=\frac{1+\sqrt{5}}{2}=1+\frac{1}{1+\frac{1}{1+\frac{1}{1+\frac{1}{1+\frac{1}{1+\begin{aligned}\ddots\end{aligned}
}}}}},\quad\sqrt{2}=1+\frac{1}{2+\frac{1}{2+\frac{1}{2+\frac{1}{2+\frac{1}{2+\begin{aligned}\ddots\end{aligned}
}}}}},\quad e=2+\frac{1}{1+\frac{1}{2+\frac{1}{1+\frac{1}{1+\frac{1}{4+\frac{1}{1+\begin{aligned}\ddots\end{aligned}
}}}}}}.
\end{align*}
This type of expansions have the extra property that they are very
easy to write. For example, in $\varphi$ and $\sqrt{2}-1$ the coefficients
are constant (1 and 2 respectively), and generally in any algebraic
number of degree 2 the coefficients will be eventually periodic (and
we go over some of these details in \subsecref{Constant-continued-fraction}).
On the other hand $e-2$ is not algebraic, but still the coefficients
are in essence periodic: they are 3-periodic where the coefficients
in the $k$-th period are $\left(1,2k,1\right)$.

However, in general this is far from true, even for well known and
useful constants, e.g. for $\pi$ we have
\[
\pi=3+\frac{1}{7+\frac{1}{15+\frac{1}{1+\frac{1}{292+\frac{1}{1+\frac{1}{1+\frac{1}{1+\frac{1}{2+\begin{aligned}\ddots\end{aligned}
}}}}}}}}.
\]
The coefficients in the expansion don't seem to have any simple pattern
that we can recognize, thus making it harder to study $\pi$ through
this expansion.

\medskip{}

The goal of this paper is to study a close relative of the simple
continued fractions called the \textbf{polynomial continued fraction}.
In this new expansion, we don't restrict the numerators to be $1$'s
as in the simple continued fraction version, however we do require
both the numerators and denominators to be polynomials. For example,
$\pi$ has this much simpler polynomial continued fraction expansion:

\[
\pi+3=6+\frac{1^{2}}{6+\frac{2^{2}}{6+\frac{3^{2}}{6+\frac{4^{2}}{6+\begin{aligned}\ddots\end{aligned}
}}}}.
\]

This new type of expansion loses some of the advantages of the simple
continued fraction expansion, and their direct relations to irrationality,
but in return we gain many more expansions, which can be much easier
to work with, thus may lead to many interesting new results. In particular,
the polynomial continued fractions are closely related to recurrence
relations with polynomial coefficients. This type of polynomial recurrence
was one of the key ingredients used by Ap\'ery to show the irrationality
of $\zeta\left(3\right)$ \cite{apery_irrationalite_1979,van_der_poorten_proof_1979}.

\medskip{}

One of the simplest way to construct polynomial continued fractions
was given by Euler, which found a way to convert many interesting
infinite sums presentations into polynomial continued fractions (see
\subsecref{Euler's-formula} for details). Our main result in this
paper is describing the other direction, namely finding simple conditions
for when a polynomial continued fraction can be converted back into
an infinite sum. We call this family the \textbf{Euler continued fractions}.
Using the notation
\[
\KK_{1}^{\infty}\frac{b_{i}}{a_{i}}:=\frac{b_{1}}{a_{1}+\cfrac{b_{2}}{a_{2}+\cfrac{b_{3}}{a_{3}+\ddots}}},
\]
we have the following:
\begin{thm}[\textbf{\textit{Euler continued fractions}}]
\label{thm:recurrence-roots}Let $h_{1},h_{2},f:\CC\to\CC$ be any
functions, and define $a,b:\CC\to\CC$ such that 
\begin{align*}
b\left(x\right) & =-h_{1}\left(x\right)h_{2}\left(x\right)\\
f\left(x\right)a\left(x\right) & =f\left(x-1\right)h_{1}\left(x\right)+f\left(x+1\right)h_{2}\left(x+1\right).
\end{align*}
Then
\[
\KK_{1}^{n}\frac{b\left(i\right)}{a\left(i\right)}=\frac{f\left(1\right)h_{2}\left(1\right)}{f\left(0\right)}\left(\frac{1}{\sum_{k=0}^{n}\frac{f\left(0\right)f\left(1\right)}{f\left(k\right)f\left(k+1\right)}\prod_{i=1}^{k}\left(\frac{h_{1}\left(i\right)}{h_{2}\left(i+1\right)}\right)}-1\right).
\]
\end{thm}

\begin{rem}
Special cases of these continued fractions appear in several places
(in particular, see for example \cite{brier_note_2022,bowman_polynomial_2018,cohen_elementary_2022}),
though we did not see this exact formulation in the literature. 
\end{rem}

In general, given polynomials $b\left(x\right),a\left(x\right)$,
it is not obvious if and how to find polynomial $h_{1}\left(x\right),h_{2}\left(x\right),f\left(x\right)$
which satisfy the conditions in \thmref{recurrence-roots}. Assuming
we know how to find all decompositions of $b\left(x\right)$, this
problem is easy to solve in the ``trivial'' case when $f\equiv1$.
More generally, if we know the degree $d_{f}=\deg\left(f\right)$,
then finding if there is a solution to 
\[
f\left(x\right)a\left(x\right)=f\left(x-1\right)h_{1}\left(x\right)+f\left(x+1\right)h_{2}\left(x+1\right)
\]
given $a\left(x\right),h_{1}\left(x\right),h_{2}\left(x\right)$ is
just a system of linear equations which is easy to solve. Our second
result is an algorithm which finds all the possible degree $d_{f}$
thus allowing to give a full algorithm checking whether a polynomial
continued fraction is in the Euler family (given that we know how
to decompose $b\left(x\right)$). We describe this algorithm in \appref{Identifying-polynomial-continued},
and in particular in \thmref{f_algorithm}, and we also provide a
python program for this algorithm in \cite{ramanujan_machine_research_group_ramanujan_2023}.

Finally, we describe a more general framework for the polynomial continued
fractions in \secref{The-most-general}, by moving to general polynomial
matrices. This led us eventually to a new structure, which we call
the\textbf{ conservative matrix field}, used to combine many polynomial
continued fractions together in order to study their limits (e.g.
if they are irrational or not). This conservative matrix field can
be used to reprove and further understand Ap\'ery's original proof
of irrationality of $\zeta\left(3\right)$, and we investigate this
structure in a later paper.

\newpage{}

\section{\label{sec:Generalized-continued-fractions}The polynomial continued
fraction}

\subsection{\label{subsec:The-definitions}Definitions}

We start with a generalization of the simple continued fractions,
which unsurprisingly, is called generalized continued fractions. These
can be defined over any topological field, though here for simplicity
we focus on the complex field with its standard Euclidean metric,
and more specifically when the numerators and denominators are integers.
\begin{defn}[\textbf{(Generalized) continued fractions}]
Let $a_{n},b_{n}$ be a sequence of complex numbers. We will write
\[
a_{0}+\KK_{1}^{n}\frac{b_{i}}{a_{i}}:=a_{0}+\cfrac{b_{1}}{a_{1}+\cfrac{b_{2}}{a_{2}+\cfrac{b_{3}}{\ddots+\cfrac{b_{n}}{a_{n}+0}}}}\in\CC\cup\left\{ \infty\right\} ,
\]
and if the limit as $n\to\infty$ exists, we also will write
\[
a_{0}+\KK_{1}^{\infty}\frac{b_{i}}{a_{i}}:=\limfi n\left(a_{0}+\KK_{1}^{n}\frac{b_{i}}{a_{i}}\right).
\]

We call this type of expansion (both finite and infinite) a \textbf{continued
fraction expansion. }In case that $a_{0}+\KK_{1}^{\infty}\frac{b_{i}}{a_{i}}$
exists, we call the finite part $a_{0}+\KK_{1}^{n-1}\frac{b_{i}}{a_{i}}$
the \textbf{$n$-th convergents} for that expansion.

A continued fraction is called \textbf{simple continued fraction},
if $b_{i}=1$ for all $i$, $a_{0}\in\ZZ$ and $1\leq a_{i}\in\ZZ$
are positive integers for $i\geq1$.

A continued fraction is called \textbf{polynomial continued fraction},
if $b_{i}=b\left(i\right),\;a_{i}=a\left(i\right)$ for some polynomials
$a\left(x\right),b\left(x\right)\in\CC\left[x\right]$ and all $i\geq1$.
\end{defn}

\begin{rem}
In these notes, when talking about \textbf{simple continued fractions}
we will always add the ``\textbf{simple}'' adjective, unlike in
most of the literature, where they are only called \textbf{continued
fractions}.
\end{rem}

\medskip{}

Simple continued fraction expansion is one of the main and basic tools
used in number theory when studying rational approximations of numbers
(for more details, see chapter 3 in \cite{einsiedler_ergodic_2013}).
The coefficients $a_{i}$ in that expansion can be found using a generalized
Euclidean division algorithm, and it is well known that a number is
rational if and only if its simple continued fraction expansion is
finite. However, while we have an algorithm to find the (almost) unique
expansion, in general they can be very complicated without any known
patterns, even for ``nice'' numbers, for example:

\[
\pi=[3;7,15,1,292,1,1,1,2,1,3,1,14,2,1,1,2,2,2,2,1,84,...]=3+\cfrac{1}{7+\cfrac{1}{15+\cfrac{1}{1+\ddots}}}.
\]

When moving to generalized continued fractions, even when we assume
that both $a_{i}$ and $b_{i}$ are integers, we lose the uniqueness
property, and the rational if and only if finite property. What we
gain in return are more presentations for each number, where some
of them can be much simpler to use. For example, $\pi$ can be written
as 
\[
\pi=3+\KK_{1}^{\infty}\frac{\left(2n-1\right)^{2}}{6}.
\]
We want to study these presentations, and (hopefully) use them to
show interesting properties, e.g. prove irrationality for certain
numbers.

\medskip{}

\subsection{The Mobius action and recurrence relations}

One of the main tools used to study continued fractions are \textbf{Mobius
transformations}.
\begin{defn}[Mobius Action]
Given a $2\times2$ invertible matrix $M=\left(\begin{smallmatrix}a & b\\
c & d
\end{smallmatrix}\right)\in\mathrm{GL_{2}\left(\CC\right)}$ and a $z\in\CC\cup\left\{ \infty\right\} $, the \textbf{Mobius action}
is defined by
\[
M\left(z\right):=\begin{cases}
\frac{az+b}{cz+d} & z\neq-\frac{d}{c},\infty\\
\infty & z=-\frac{d}{c}\\
\frac{a}{c} & z=\infty
\end{cases}.
\]

In other words, we apply the standard matrix multiplication $\left(\begin{smallmatrix}a & b\\
c & d
\end{smallmatrix}\right)\left(\begin{smallmatrix}z\\
1
\end{smallmatrix}\right)=\left(\begin{smallmatrix}az+b\\
cz+d
\end{smallmatrix}\right)$ and project it onto $R^{1}\CC$ by dividing the $x$-coordinate by
the $y$-coordinate. As scalars act trivially, we get an action of
$\mathrm{PGL}_{2}\left(\CC\right)\cong\GL_{2}\left(\CC\right)/\CC^{\times}$.
\end{defn}

By this definition, it is easy to see that 
\[
\KK_{1}^{n}\frac{b_{i}}{a_{i}}=\frac{b_{1}}{a_{1}+\frac{b_{2}}{a_{2}+\frac{\ddots}{\frac{b_{n}}{a_{n}+0}}}}=\left(\begin{smallmatrix}0 & b_{1}\\
1 & a_{1}
\end{smallmatrix}\right)\left(\begin{smallmatrix}0 & b_{2}\\
1 & a_{2}
\end{smallmatrix}\right)\cdots\left(\begin{smallmatrix}0 & b_{n}\\
1 & a_{n}
\end{smallmatrix}\right)\left(0\right).
\]
More over, when the continued fraction converges, we get that
\[
\KK_{1}^{\infty}\frac{b_{i}}{a_{i}}=\left(\begin{smallmatrix}0 & b_{1}\\
1 & a_{1}
\end{smallmatrix}\right)\left(\KK_{2}^{\infty}\frac{b_{i}}{a_{i}}\right).
\]

This Mobius presentation allows us to show an interesting recurrence
relation on the numerators and denominators of the convergents, which
generalizes the well known recurrence on simple continued fractions.
\begin{lem}
\label{lem:gcf-recursion}Let $a_{n},b_{n}$ be a sequence of integers.
Define $M_{n}=\left(\begin{smallmatrix}0 & b_{n}\\
1 & a_{n}
\end{smallmatrix}\right)$ and set $\left(\begin{smallmatrix}p_{n}\\
q_{n}
\end{smallmatrix}\right)=\left(\prod_{1}^{n-1}M_{i}\right)\left(\begin{smallmatrix}0\\
1
\end{smallmatrix}\right)$. Then $\frac{p_{n}}{q_{n}}=\prod_{1}^{n-1}M_{i}\left(0\right)=\KK_{1}^{n-1}\frac{b_{i}}{a_{i}}$
are the convergents of the generalized continued fraction presentation.
More over, we have that $\left(\begin{smallmatrix}p_{n-1} & p_{n}\\
q_{n-1} & q_{n}
\end{smallmatrix}\right)=\prod_{1}^{n-1}M_{i}$, implying the same recurrence relation on $p_{n}$ and $q_{n}$ given
by 
\begin{align*}
p_{n+1} & =a_{n}p_{n}+b_{n}p_{n-1}\\
q_{n+1} & =a_{n}q_{n}+b_{n}q_{n-1},
\end{align*}
with starting condition $p_{0}=1,\;p_{1}=0$ and $q_{0}=0,\;q_{1}=1$.
\end{lem}

\begin{proof}
Our definition of $\left(\begin{smallmatrix}p_{n}\\
q_{n}
\end{smallmatrix}\right)=\left(\prod_{1}^{n-1}M_{i}\right)\left(\begin{smallmatrix}0\\
1
\end{smallmatrix}\right)$ simply gives us the right column of $\prod_{1}^{n-1}M_{i}$. Since
$M_{i}\left(\begin{smallmatrix}1\\
0
\end{smallmatrix}\right)=\left(\begin{smallmatrix}0\\
1
\end{smallmatrix}\right)$, we see that 
\[
\left(\prod_{1}^{n-1}M_{i}\right)\left(\begin{smallmatrix}1\\
0
\end{smallmatrix}\right)=\left(\prod_{1}^{n-2}M_{i}\right)\left(\begin{smallmatrix}0\\
1
\end{smallmatrix}\right)=\left(\begin{smallmatrix}p_{n-1}\\
q_{n-1}
\end{smallmatrix}\right)\quad\forall n\geq2,
\]
and for $n=1$ we have $\overbrace{\left(\prod_{1}^{n-1}M_{i}\right)}^{=Id}\left(\begin{smallmatrix}1\\
0
\end{smallmatrix}\right)=\left(\begin{smallmatrix}p_{0}\\
q_{0}
\end{smallmatrix}\right)$ , so together we have that 
\[
\prod_{1}^{n-1}M_{i}=\left(\begin{smallmatrix}p_{n-1} & p_{n}\\
q_{n-1} & q_{n}
\end{smallmatrix}\right)\quad\forall n\geq1.
\]

From this equation we get the matrix recurrence $\left(\begin{smallmatrix}p_{n-1} & p_{n}\\
q_{n-1} & q_{n}
\end{smallmatrix}\right)M_{n}=\left(\begin{smallmatrix}p_{n} & p_{n+1}\\
q_{n} & q_{n+1}
\end{smallmatrix}\right)$ , implying the same recurrence on $p_{n},q_{n}$ .
\end{proof}

\newpage{}

\section{\label{sec:Finding-the-limit}Finding the limit of continued fractions}

The Euclidean division algorithm allows us to determine the (almost)
unique simple continued fraction expansion of any given number. However,
this unique expansion is many times difficult to write, as there doesn't
seem to be any pattern for the coefficients appearing in the expansion.
On the other hand, generalized continued fractions lack this unique
expansion property, rendering a single algorithm impractical, but
in many interesting cases have simple patterns which can be exploited.

In this section we describe the Euler polynomial continued fraction
derived from Euler's conversion which have nice polynomial patterns
on the one hand, and on the other hand are easy to compute in many
cases.

\subsection{\label{subsec:Constant-continued-fraction}Constant continued fraction}

Before describing the Euler polynomial continued fraction, let us
get some intuition from its simplest case, where the polynomials are
actually constant. This is not necessary for the proofs in the next
section, and is only used as a motivation.

Recall that a simple continued fraction is \textbf{eventually periodic}
if the coefficients $a_{i}$ in the expansion
\[
\frac{1}{a_{1}+\frac{1}{a_{2}+\frac{1}{a_{3}+\begin{aligned}\ddots\end{aligned}
}}}
\]
are eventually periodic, namely there is some positive integer $T$
such that $a_{i}=a_{i+T}$ for all $i$ large enough. It is well known
that the simple continued fraction of $\alpha$ is eventually periodic
if and only if $\alpha$ is algebraic of degree $2$, namely a root
of a polynomial in $\QQ\left[x\right]$ of degree $2$ (see chapter
3 in \cite{einsiedler_ergodic_2013}). In our more generalized setting,
this periodicity could be simplified to the case where $T=1$, to
obtain the following:
\begin{claim}
Let $A,B\in\RR$ with $B\neq0$. 
\begin{enumerate}
\item If $\alpha=\KK_{1}^{\infty}\frac{B}{A}$ converges to a finite number,
then $A\neq0$, $A^{2}+4B\geq0$ and $\alpha$ is a root for
\begin{equation}
x^{2}+Ax-B=0.\label{eq:constant_cf_root}
\end{equation}
\item On the other hand, if $A\neq0,\;A^{2}+4B\geq0$, then the two real
roots of \eqref{constant_cf_root} are exactly
\[
\KK_{1}^{\infty}\frac{B}{A}=\cfrac{B}{A+\cfrac{B}{A+\cfrac{B}{A+\ddots}}}\quad;\quad-A-\KK_{1}^{\infty}\frac{B}{A}=-\left[A+\cfrac{B}{A+\cfrac{B}{A+\cfrac{B}{A+\ddots}}}\right].
\]
\end{enumerate}
\end{claim}

\begin{proof}
\begin{enumerate}
\item Let $S=\begin{pmatrix}0 & B\\
1 & A
\end{pmatrix}$, and suppose that $\alpha=\limfi nS^{n}\left(0\right)=\KK_{1}^{\infty}\frac{B}{A}$
converges. Note that $A\neq0$, since otherwise $S^{n}\left(0\right)$
would alternate between $0$ and $\infty$.\\
Since $x\mapsto S\left(x\right)$ is continues, we get that $\alpha=S\left(\alpha\right)$
is a fixed point and therefore $\frac{B}{\alpha+A}=\alpha$. By definition,
for $\alpha=\infty$ the two sides are $0\neq\infty$ and for $\alpha=-A$
the two sides are $\infty\neq-A$, so we conclude that $\alpha\neq\infty,-A$.
Hence, we can multiply both sides by $\left(\alpha+A\right)$ to get
that 
\[
\alpha^{2}+A\alpha-B=0.
\]
As $\alpha$ is a limit of real numbers, it is real, and therefore
the discriminant $A^{2}+4B\geq0$ is nonnegative. 
\item The other direction is a standard exercise in linear algebra using
the diagonalization or the Jordan form of (almost) $S$. Indeed, letting
$\alpha_{\pm}=\frac{A\pm\sqrt{A^{2}+4B}}{2}$ be the eigenvalues of
$S$, we can write it as 
\[
S=\begin{cases}
\begin{pmatrix}\alpha_{-} & \alpha_{+}\\
-1 & -1
\end{pmatrix}\begin{pmatrix}\alpha_{+} & 0\\
0 & \alpha_{-}
\end{pmatrix}\begin{pmatrix}\alpha_{-} & \alpha_{+}\\
-1 & -1
\end{pmatrix}^{-1} & A^{2}+4B>0\\
\begin{pmatrix}A/2 & 0\\
0 & 1
\end{pmatrix}\begin{pmatrix}-1 & 1\\
1 & 0
\end{pmatrix}\begin{pmatrix}1 & 1\\
0 & 1
\end{pmatrix}\begin{pmatrix}-1 & 1\\
1 & 0
\end{pmatrix}^{-1}\begin{pmatrix}1 & 0\\
0 & A/2
\end{pmatrix} & A^{2}+4B=0.
\end{cases}
\]
The rest is a standard analysis of $S^{n}\begin{pmatrix}0\\
1
\end{pmatrix}$. It uses the standard trick of $\left(PMP^{-1}\right)^{n}=PM^{n}P^{-1}$
and the second Mobius trick where
\[
\left[\begin{pmatrix}A/2 & 0\\
0 & 1
\end{pmatrix}M\begin{pmatrix}1 & 0\\
0 & A/2
\end{pmatrix}\right]^{n}\left(0\right)=\begin{pmatrix}1 & 0\\
0 & 2/A
\end{pmatrix}\left[\begin{pmatrix}A/2 & 0\\
0 & A/2
\end{pmatrix}M\right]^{n}\begin{pmatrix}1 & 0\\
0 & A/2
\end{pmatrix}\left(0\right)=\begin{pmatrix}1 & 0\\
0 & 2/A
\end{pmatrix}M^{n}\left(0\right),
\]
namely that as Mobius transformations the diagonal matrices fix $0$
and scalar matrices fix any number. We leave the rest as an exercise
to the reader.
\end{enumerate}
\end{proof}
To summarize the results so far, in order to understand $\KK_{1}^{\infty}\frac{B}{A}$,
namely the continued fraction arising from the matrix $S=\begin{pmatrix}0 & B\\
1 & A
\end{pmatrix}$, we need to find solutions for $\alpha_{+}+\alpha_{-}=-A$ and $\alpha_{+}\alpha_{-}=-B$,
namely the roots for characteristic polynomial of $(-S)$, namely
$x^{2}+Ax-B=0$. Also, along the proof we saw that some continued
fractions expansions can be ``simplified'' by using the fact that
scalar matrices act trivially. In general, our matrix will be a non
constant polynomial, however as we shall see, it still has a similar
behavior to the constant matrix type.

\subsection{\label{subsec:Euler's-formula}Euler's formula}

One of the most elementary and useful continued fraction presentation
was introduced by Euler who found a way to convert standard finite
sums (and their infinite sum limits) to generalized continued fraction.
\begin{thm}[Euler's formula]
 Let $r_{i}\in\CC$ for $i\geq1$. Then
\[
1+r_{1}+r_{1}r_{2}+\cdots+r_{1}\cdots r_{n}=\sum_{k=0}^{n}\left(\prod_{i=1}^{k}r_{i}\right)=\frac{1}{1+\KK_{1}^{n}\frac{-r_{i}}{1+r_{i}}}.
\]
By taking the limit (if exists), we have that 
\[
\KK_{1}^{\infty}\frac{-r_{i}}{1+r_{i}}=\frac{1}{\sum_{k=0}^{\infty}\prod_{i=1}^{k}r_{i}}-1.
\]
\end{thm}

\begin{proof}
This is a standard induction, which we leave as an exercise to the
reader.
\end{proof}
For more details and applications of this formula, the reader is referred
to \cite{jones_continued_1980}.

Euler's formula implies that whenever $a_{i}+b_{i}=\left(1+r_{i}\right)+\left(-r_{i}\right)=1$
we can go back from generalized continued fractions $\KK_{1}^{\infty}\frac{b_{i}}{a_{i}}$
to infinite sums, where we have many more tools at our disposal. However,
in general this condition doesn't hold, though fortunately Euler had
a trick to move to equivalent expansion where it might hold (which
we already saw in the constant continued fractions case).

\newpage{}
\begin{lem}[The equivalence transformation]
\label{lem:equivalence-transformation} Let $a_{i},b_{i}\in\CC$
be two sequences and $0\neq c_{i}\in\CC$ another sequence with nonzero
elements. Then
\[
\KK_{1}^{n}\frac{b_{i}}{a_{i}}=\frac{1}{c_{0}}\KK_{1}^{n}\frac{c_{i-1}c_{i}b_{i}}{c_{i}a_{i}}.
\]
\end{lem}

\begin{proof}
Intuitively, this lemma follows from the fact that $\frac{b_{i}}{a_{i}+x}=\frac{c_{i}b_{i}}{c_{i}a_{i}+c_{i}x}$
plus induction. For example
\[
\frac{4}{11}=\frac{1}{2+\frac{1}{1+\frac{1}{3}}}=\frac{1}{{\color{red}\boldsymbol{c_{0}}}}\cdot\frac{{\color{red}\boldsymbol{c_{0}}}}{2+\frac{1}{1+\frac{1}{3}}}=\frac{1}{c_{0}}\cdot\frac{c_{0}{\color{red}\boldsymbol{c_{1}}}}{2{\color{red}\boldsymbol{c_{1}}}+\frac{{\color{red}\boldsymbol{c_{1}}}}{1+\frac{1}{3}}}=\frac{1}{c_{0}}\cdot\frac{c_{0}c_{1}}{2c_{1}+\frac{c_{1}{\color{red}\boldsymbol{c_{2}}}}{{\color{red}\boldsymbol{c_{2}}}+\frac{{\color{red}\boldsymbol{c_{2}}}}{3}}}=\frac{1}{c_{0}}\cdot\frac{c_{0}c_{1}}{2c_{1}+\frac{c_{1}c_{2}}{c_{2}+\frac{c_{2}{\color{red}\boldsymbol{c_{3}}}}{3{\color{red}\boldsymbol{c_{3}}}}}}.
\]
More precisely, recall that $\KK_{1}^{n}\frac{b_{i}}{a_{i}}=\left[\prod_{1}^{n}\left(\begin{smallmatrix}0 & b_{i}\\
1 & a_{i}
\end{smallmatrix}\right)\right]\left(0\right)$. As Mobius transformations defined by scalar matrices are the identity,
we get that
\begin{align*}
\KK_{1}^{n}\frac{b_{i}}{a_{i}}=\left[\prod_{1}^{n}\left(\begin{smallmatrix}0 & b_{i}\\
1 & a_{i}
\end{smallmatrix}\right)\right]\left(0\right) & =\left[\prod_{1}^{n}\left(\begin{smallmatrix}0 & b_{i}\\
1 & a_{i}
\end{smallmatrix}\right)\cdot c_{i}I\right]\left(0\right)=\left[\prod_{1}^{n}\left(\begin{smallmatrix}0 & b_{i}\\
1 & a_{i}
\end{smallmatrix}\right)\cdot\left(\begin{smallmatrix}1 & 0\\
0 & c_{i}
\end{smallmatrix}\right)\left(\begin{smallmatrix}c_{i} & 0\\
0 & 1
\end{smallmatrix}\right)\right]\left(0\right)\\
 & =\left(\begin{smallmatrix}c_{0}^{-1} & 0\\
0 & 1
\end{smallmatrix}\right)\left[\prod_{1}^{n}\left(\left(\begin{smallmatrix}c_{i-1} & 0\\
0 & 1
\end{smallmatrix}\right)\left(\begin{smallmatrix}0 & b_{i}\\
1 & a_{i}
\end{smallmatrix}\right)\left(\begin{smallmatrix}1 & 0\\
0 & c_{i}
\end{smallmatrix}\right)\right)\right]\left(\begin{smallmatrix}c_{n} & 0\\
0 & 1
\end{smallmatrix}\right)\left(0\right)\\
 & =\left(\begin{smallmatrix}c_{0}^{-1} & 0\\
0 & 1
\end{smallmatrix}\right)\left[\prod_{1}^{n}\left(\begin{smallmatrix}0 & c_{i-1}c_{i}b_{i}\\
1 & c_{i}a_{i}
\end{smallmatrix}\right)\right]\left(\begin{smallmatrix}c_{n} & 0\\
0 & 1
\end{smallmatrix}\right)\left(0\right)
\end{align*}
Since $\left(\begin{smallmatrix}c_{n} & 0\\
0 & 1
\end{smallmatrix}\right)\left(\begin{smallmatrix}0\\
1
\end{smallmatrix}\right)=\left(\begin{smallmatrix}0\\
1
\end{smallmatrix}\right)$ , we conclude that
\[
\KK_{1}^{n}\frac{b_{i}}{a_{i}}=\frac{1}{c_{0}}\KK_{1}^{n}\frac{c_{i-1}c_{i}b_{i}}{c_{i}a_{i}}.
\]
\end{proof}
\begin{rem}
In the last lemma we basically moved from the matrices $\left(\begin{smallmatrix}0 & b_{i}\\
1 & a_{i}
\end{smallmatrix}\right)$ to $\left(c_{n-1}U_{n-1}\right)^{-1}M_{n}U_{n}$ with $U_{n}=\left(\begin{smallmatrix}1 & 0\\
0 & c_{n}
\end{smallmatrix}\right)$. This type of equivalence can be generalized, as we shall see it
later in \secref{The-most-general}.
\end{rem}

\medskip{}

Combining the last lemma and Euler's formula, we are led to look for
$c_{n}$ satisfying
\[
c_{n}a_{n}+c_{n-1}c_{n}b_{n}=1.
\]
If we can find such $c_{n}$, then we have the following.
\begin{cor}
\label{cor:c-n-conversion}Let $a_{i},b_{i}\in\CC$ be any sequences
and suppose that we can find a solution to 
\[
c_{i}a_{i}+c_{i-1}c_{i}b_{i}=1
\]
with nonzero $c_{i}$. Then
\[
\KK_{1}^{n}\frac{b_{i}}{a_{i}}=\frac{1}{c_{0}}\KK_{1}^{n}\frac{\left(c_{i-1}c_{i}b_{i}\right)}{\left(c_{i}a_{i}\right)}=\frac{1}{c_{0}}\left(\frac{1}{\sum_{k=0}^{n}\left(-1\right)^{k}\prod_{i=1}^{k}\left(c_{i-1}c_{i}b_{i}\right)}-1\right),
\]
or equivalently
\[
\sum_{k=0}^{n}\left(-1\right)^{k}\prod_{i=1}^{k}\left(c_{i-1}c_{i}b_{i}\right)=\frac{1}{1+c_{0}\cdot\KK_{1}^{n}\frac{b_{i}}{a_{i}}}=\left(\begin{smallmatrix}0 & 1\\
1 & c_{0}
\end{smallmatrix}\right)\left(\KK_{1}^{n}\frac{b_{i}}{a_{i}}\right).
\]

\newpage{}
\end{cor}

\begin{example}[The exponential function]
\label{exa:(The-exponential-function)} Given some $x\in\RR$, we start with the standard Taylor expansion
for $e^{x}$:
\[
e^{x}=1+x+\frac{x^{2}}{2}+\frac{x^{3}}{3!}+\cdots=\sum_{0}^{\infty}\frac{x^{n}}{n!}=\sum_{0}^{\infty}\prod_{1}^{n}\left(\frac{x}{i}\right).
\]
Taking $r_{i}=\frac{x}{i}$ in Euler's formula we get that 
\[
e^{x}=\frac{1}{1+\KK_{1}^{\infty}\frac{-x/i}{1+x/i}}.
\]
We would like to use the equivalence transformation with $c_{i}=i$
so as to remove the division in the numerators and denominators, however
since $c_{0}=0$ we cannot directly do it. Instead, we will apply
it starting from the second index, namely
\begin{align*}
\KK_{i=1}^{\infty}\frac{-x/i}{1+x/i} & =\frac{-x}{1+x+\KK_{i=2}^{\infty}\frac{-x/i}{1+x/i}}=\frac{-x}{1+x+\frac{1}{1}\KK_{2}^{\infty}\frac{-x\left(i-1\right)}{i+x}}=\frac{-x}{1+x+\KK_{1}^{\infty}\frac{-xi}{1+i+x}},
\end{align*}
so that 
\[
e^{x}=\frac{1}{1-\frac{x}{1+x-\frac{x}{2+x-\frac{2x}{3+x-\frac{3x}{4+x-\frac{4x}{\ddots}}}}}}.
\]

In particular, for $x=1$ and $x=-1$ we get that 
\[
\frac{2-e}{e-1}=\KK_{1}^{\infty}\frac{-i}{2+i}\quad,\quad\frac{1}{e-1}=\KK_{1}^{\infty}\frac{i}{i}.
\]

Similar computation can be done with other functions like $\sin\left(x\right),\cos\left(x\right),\ln\left(1+x\right)$
etc.\medskip{}
\end{example}

Finding $c_{n}$ which satisfy the relation in \corref{c-n-conversion}
above is equivalent to solving the recurrence
\[
c_{i}:=\frac{1}{a_{i}+c_{i-1}b_{i}}=\left(\begin{smallmatrix}0 & 1\\
b_{i} & a_{i}
\end{smallmatrix}\right)\left(c_{i-1}\right).
\]

Of course, the hard part is not to find some sequence $c_{i}$ satisfying
this relation, but a ``nice enough'' such sequence for which we
can compute $\sum_{k=0}^{n}\left(-1\right)^{k}\prod_{i=1}^{k}\left(c_{i-1}c_{i}b_{i}\right)$.
The matrix above is the transpose of our original polynomial continued
fraction matrix. Transposing it, and thinking of $c_{i}=\frac{F_{i}}{F_{i+1}}$
as the projection of $\left(F_{i},F_{i+1}\right)$, we get the recursion:
\[
\left(\begin{smallmatrix}F_{i-1} & F_{i}\end{smallmatrix}\right)\left(\begin{smallmatrix}0 & b_{i}\\
1 & a_{i}
\end{smallmatrix}\right)=\left(\begin{smallmatrix}F_{i} & F_{i+1}\end{smallmatrix}\right)\quad\iff\quad F_{i-1}b_{i}+F_{i}a_{i}=F_{i+1}.
\]
This is exactly the recurrence satisfied by $p_{n}$ and $q_{n}$
we saw in \lemref{gcf-recursion} (so that both $c_{i}=\frac{p_{i}}{p_{i+1}}$
and $c_{i}=\frac{q_{i}}{q_{i+1}}$ solve the recurrence above). Thus,
in a sense, finding one such ``nice'' solution to the recursion,
let us find both $p_{n}$ and $q_{n}$.

Next, we try to give simple conditions on $a_{i},b_{i}$ where we
can find a ``nice'' solution for the recurrence above. We have already
seen in \subsecref{Constant-continued-fraction} one such case where
both $a_{i}\equiv A,\;b_{i}\equiv B$ are fixed, so that our matrix
is $M_{i}=M=\left(\begin{smallmatrix}0 & B\\
1 & A
\end{smallmatrix}\right)$, and then $\left(F_{n},F_{n+1}\right)=\left(F_{0},F_{1}\right)M^{n}$.
In this simple recursion, we used the structure of $M$, whether it
is diagonalizable or not, in order to find $F_{n}$. This has led
us to look for solutions $h_{1},h_{2}$ of $x^{2}+Ax-B=0$, namely
$h_{1}\cdot h_{2}=-B$ and $h_{1}+h_{2}=-A$ and then consider their
$n$-th powers. In the case where the roots are the same, instead
of just taking $h_{1}^{n}=h_{2}^{n}$, there is also a polynomial
involved in the solutions.

More generally, our recurrence depends on $i$, and the ``right''
way to think about exponential is more like factorial, so we should
look for $F_{n}$ of the form $f\left(n\right)\cdot\prod_{1}^{n}h\left(k\right)$
for some polynomials $f,h$ , or in the $c_{n}$ notation we have
$c_{n}=\frac{F_{n}}{F_{n+1}}=\frac{f\left(n\right)}{f\left(n+1\right)h\left(n+1\right)}$.

With this quadratic intuition in mind, we restate and prove \thmref{recurrence-roots}

\newpage{}

\setcounter{thm}{0}
\begin{thm}[Euler continued fractions]
\footnote{It has come to our attention that Euler doesn't have enough mathematical
objects named after him.}Let $h_{1},h_{2},f:\CC\to\CC$ be any functions, and define $a,b:\CC\to\CC$
such that 
\begin{align*}
b\left(x\right) & =-h_{1}\left(x\right)h_{2}\left(x\right)\\
f\left(x\right)a\left(x\right) & =f\left(x-1\right)h_{1}\left(x\right)+f\left(x+1\right)h_{2}\left(x+1\right)
\end{align*}
Then taking $F_{n}=f\left(n\right)\cdot\prod_{1}^{n}h_{2}\left(k\right)$
solves the recurrence relation
\[
F_{n-1}b\left(n\right)+F_{n}a\left(n\right)=F_{n+1},
\]
and we get that
\[
\KK_{1}^{n}\frac{b\left(i\right)}{a\left(i\right)}=\frac{f\left(1\right)h_{2}\left(1\right)}{f\left(0\right)}\left(\frac{1}{\sum_{k=0}^{n}\frac{f\left(0\right)f\left(1\right)}{f\left(k\right)f\left(k+1\right)}\prod_{i=1}^{k}\left(\frac{h_{1}\left(i\right)}{h_{2}\left(i+1\right)}\right)}-1\right).
\]
\end{thm}

\begin{proof}
Simply putting the definition of $a,b,F$ in the recurrence gives
us
\[
F_{n}a\left(n\right)+F_{n-1}b\left(n\right)-F_{n+1}=\left[\prod_{1}^{n}h_{2}\left(k\right)\right]\left(f\left(n\right)a\left(n\right)-f\left(n-1\right)h_{1}\left(n\right)-f\left(n+1\right)h_{2}\left(n+1\right)\right)=0.
\]
Using \corref{c-n-conversion} and taking $c_{n}=\frac{F_{n}}{F_{n+1}}$
we get that
\[
\KK_{1}^{n}\frac{b\left(i\right)}{a\left(i\right)}=\frac{F_{1}}{F_{0}}\left(\frac{1}{\sum_{k=0}^{n}\prod_{i=1}^{k}\left(\frac{F_{i-1}}{F_{i+1}}h_{1}\left(i\right)h_{2}\left(i\right)\right)}-1\right)=\frac{f\left(1\right)h_{2}\left(1\right)}{f\left(0\right)}\left(\frac{1}{\sum_{k=0}^{n}\frac{f\left(0\right)f\left(1\right)}{f\left(k\right)f\left(k+1\right)}\prod_{i=1}^{k}\left(\frac{h_{1}\left(i\right)}{h_{2}\left(i+1\right)}\right)}-1\right).
\]
\end{proof}
\setcounter{thm}{12}
\begin{rem}
We leave it as an exercise to show that solutions to $c\left(n\right)a\left(n\right)+c\left(n-1\right)c\left(n\right)b\left(n\right)=1$
where $a,b\in\CC\left[x\right]$, and $c\left(x\right)\in\CC\left(x\right)$
is a rational function, imply that $a,b$ must have the form as in
the theorem above where $f,h_{1},h_{2}$ are all polynomials in themselves
and $c\left(x\right)=\frac{f\left(x\right)}{f\left(x+1\right)h_{1}\left(x+1\right)}$.
Furthermore, in case that $a\left(x\right),b\left(x\right)$ is constant,
you should check that we get back the results from the previous section.\medskip{}
\end{rem}

While the theorem above is applicable to any functions $h_{1},h_{2}$
and $f$, it is probably most useful when they are polynomials over
$\ZZ$, in which case the product $\prod_{i=1}^{k}\left(\frac{h_{1}\left(i\right)}{h_{2}\left(i+1\right)}\right)$
can become much simpler.
\begin{example}[The trivial Euler family]
\label{exa:Euler-family} In \thmref{recurrence-roots}, consider
the case where $f\left(x\right)=1$:
\begin{align*}
b\left(x\right) & =-h_{1}\left(x\right)h_{2}\left(x\right)\\
a\left(x\right) & =h_{1}\left(x\right)+h_{2}\left(x+1\right).
\end{align*}
In this case we have 
\[
\KK_{1}^{n}\frac{b_{i}}{a_{i}}=h_{2}\left(1\right)\left(\frac{1}{\sum_{k=0}^{n}\prod_{i=1}^{k}\left(\frac{h_{1}\left(i\right)}{h_{2}\left(i+1\right)}\right)}-1\right),
\]
or alternatively
\[
\sum_{k=0}^{n}\prod_{i=1}^{k}\left(\frac{h_{1}\left(i\right)}{h_{2}\left(i+1\right)}\right)=\left(\frac{1}{h_{2}\left(1\right)}\KK_{1}^{n}\frac{b_{i}}{a_{i}}+1\right)^{-1}.
\]
\begin{enumerate}
\item Suppose that $h_{1}\left(i\right),h_{2}\left(i+1\right)\neq0$ for
$i\geq1$ and $\limfi i\frac{h_{1}\left(i\right)}{h_{2}\left(i\right)}>1$
(namely, $\deg\left(h_{1}\right)>\deg\left(h_{2}\right)$, or they
have the same degree, and the leading coefficient of $h_{1}\left(x\right)$
is larger than the leading coefficient of $h_{2}\left(x\right)$ in
absolute values). In this case, the infinite sum expression above
converge to infinity in absolute value, so that $\KK_{1}^{\infty}\frac{b_{i}}{a_{i}}=-h_{2}\left(1\right)$. 
\item Suppose that both $a\left(x\right)$ and $b\left(x\right)$ are constant.
As we have seen in \ref{subsec:Constant-continued-fraction}, it is
then easy to show that if $\KK_{1}^{\infty}\frac{b}{a}$ converges,
then it is a root of $z^{2}+az-b=0$. Let us show how to reprove this
result with \thmref{recurrence-roots}. Note that the roots $\lambda_{1},\lambda_{2}$
are exactly $-h_{1}\left(x\right),-h_{2}\left(x\right)$. As these
are interchangeable, we may assume that $\left|\lambda_{1}\right|\geq\left|\lambda_{2}\right|$.
If $\left|\lambda_{1}\right|>\left|\lambda_{2}\right|$ or $\lambda_{1}=\lambda_{2}$,
then just as in the previous case, we will get that $\KK_{1}^{\infty}\frac{b}{a}=-h_{2}\left(1\right)=\lambda_{2}$
as expected. If $\left|\lambda_{1}\right|=\left|\lambda_{2}\right|$
but $\lambda_{1}\neq\lambda_{2}$, then there is no convergence.
\end{enumerate}
In the following examples we denote $P_{k}:=\prod_{i=1}^{k}\left(\frac{h_{1}\left(i\right)}{h_{2}\left(i+1\right)}\right)$
so that $\KK_{1}^{n}\frac{b_{i}}{a_{i}}=h_{2}\left(1\right)\left(\left(\sum_{0}^{n}P_{k}\right)^{-1}-1\right)$.
\begin{enumerate}
\item [3.]\setcounter{enumi}{3}Let $h_{1}\left(x\right)=1$ and $h_{2}\left(x\right)=x$.
We then have that $P_{k}=\frac{1}{\left(k+1\right)!}$, so that
\[
\KK_{1}^{\infty}\frac{-n}{n+2}=\frac{1}{\sum_{k=0}^{\infty}\frac{1}{\left(k+1\right)!}}-1=\frac{1}{e-1}-1,
\]
which we already saw in \exaref{(The-exponential-function)}.
\item Let $h_{1}\left(x\right)=h_{2}\left(x\right)=x^{d}$ for some $d\geq2$.
Then $P_{k}=\prod_{i=1}^{k}\frac{i^{d}}{\left(i+1\right)^{d}}=\frac{1}{\left(k+1\right)^{d}}$,
so that
\[
\KK_{1}^{\infty}\frac{b_{n}}{a_{n}}=\frac{1}{\sum_{k=0}^{\infty}\frac{1}{\left(k+1\right)^{d}}}-1=\frac{1}{\zeta\left(d\right)}-1.
\]
\item Let $h_{1}\left(x\right)=x^{d-1}\left(x+2\right)$ and $h_{2}\left(x\right)=x^{d}$
for some $d\geq3$. We then have that 
\[
P_{k}=\prod_{i=1}^{k}\frac{i^{d-1}}{\left(i+1\right)^{d-1}}\cdot\prod_{i=1}^{k}\frac{i+2}{i+1}=\frac{1}{\left(k+1\right)^{d-1}}\cdot\frac{k+2}{2}=\frac{1}{2}\left[\frac{1}{\left(k+1\right)^{d-1}}+\frac{1}{\left(k+1\right)^{d-2}}\right].
\]
Summing up over $k$ gives us
\[
\sum_{k=0}^{\infty}\frac{1}{2}\left[\frac{1}{\left(k+1\right)^{d-1}}+\frac{1}{\left(k+1\right)^{d-2}}\right]=\frac{1}{2}\left(\zeta\left(d-1\right)+\zeta\left(d-2\right)\right).
\]
A similar computation can be done for $h_{1}\left(x\right)=x^{d-1}\left(x+m\right)$
and $h_{2}\left(x\right)=x^{d}$ where $m\geq0$. Note that for $m=0$
we get the continued fraction from part $\left(2\right)$ above and
for $m=1$ we can use the equivalence transformation \lemref{equivalence-transformation}
to cancel $\left(x+1\right)$ in $a\left(x\right)$ and $x\left(x+1\right)$
in $b\left(x\right)$ and get 
\[
\KK_{1}^{\infty}\frac{b\left(x\right)}{a\left(x\right)}=\KK_{1}^{\infty}\frac{-x^{2\left(d-1\right)}}{x^{\left(d-1\right)}+\left(1+x\right)^{\left(d-1\right)}}=\frac{1}{\zeta\left(d-1\right)}-1.
\]
\item Let $h_{1}\left(x\right)=x^{d}$ and $h_{2}\left(x\right)=x^{d-1}\left(x+1\right)$
for some $d\geq2$. We then have that 
\begin{align*}
P_{k} & =\prod_{i=1}^{k}\frac{i^{d-1}}{\left(i+1\right)^{d-1}}\cdot\prod_{i=1}^{k}\frac{i}{i+2}=\frac{1}{\left(k+1\right)^{d-1}}\cdot\frac{1\cdot2}{\left(k+1\right)\left(k+2\right)}=\frac{1}{\left(k+1\right)^{d}}\frac{2}{k+2}.
\end{align*}
Since $\frac{1}{j+1}\cdot\frac{1}{j+2}=\frac{1}{j+1}-\frac{1}{j+2}$
for all $j\geq0$, we get by induction that
\[
\frac{P_{k}}{2}=\frac{1}{\left(k+1\right)^{d}}\cdot\frac{1}{k+2}=\sum_{\ell=2}^{d}\frac{\left(-1\right)^{d-\ell}}{\left(k+1\right)^{\ell}}+\left(-1\right)^{d-1}\left(\frac{1}{k+1}-\frac{1}{k+2}\right).
\]
Summing over this expression we get
\begin{align*}
\sum_{k=0}^{\infty}\left[\sum_{\ell=2}^{d}\frac{\left(-1\right)^{d-\ell}}{\left(k+1\right)^{\ell}}+\left(-1\right)^{d-1}\left(\frac{1}{k+1}-\frac{1}{k+2}\right)\right] & =\sum_{\ell=2}^{d}\left(-1\right)^{d-\ell}\zeta\left(\ell\right)+\left(-1\right)^{d-1}.
\end{align*}
\end{enumerate}
\end{example}

\begin{rem}
In general, the ideas appearing in the examples above can help compute
many of the sums of the form $\sum_{k=0}^{\infty}\frac{f\left(0\right)f\left(1\right)}{f\left(k\right)f\left(k+1\right)}\prod_{i=1}^{k}\left(\frac{h_{1}\left(i\right)}{h_{2}\left(i+1\right)}\right)$
appearing in \thmref{recurrence-roots}.
\begin{itemize}
\item As the first example shows, we may assume that $\deg\left(h_{1}\right)\leq\deg\left(h_{2}\right)$.
\item Next, if $\deg\left(h_{1}\right)=\deg\left(h_{2}\right)$ and all
the roots of $h_{1},h_{2}$ are integers, then most of the elements
in $\prod_{i=1}^{k}\left(\frac{h_{1}\left(i\right)}{h_{2}\left(i+1\right)}\right)$
will be canceled out and $\frac{f\left(0\right)f\left(1\right)}{f\left(k\right)f\left(k+1\right)}\prod_{i=1}^{k}\left(\frac{h_{1}\left(i\right)}{h_{2}\left(i+1\right)}\right)$
will be a rational function. Using the standard decomposition of rational
functions, we can write it as a linear combination of the functions
of the form $\frac{1}{\left(n-\alpha\right)^{d}}$ with integer $\alpha$.
Their sum always converge if $d\geq2$ to values of the zeta function.
The elements for which $d=1$ should be put together to find out their
limits (e.g. $\sum_{0}^{\infty}\left(\frac{1}{k+1}-\frac{1}{k+2}\right)=1$).
This can be slightly generalized, if the roots of $h_{2}$ and $h_{1}$
are the same modulo the integers, in which case still most of the
elements in $\prod_{i=1}^{k}\left(\frac{h_{1}\left(i\right)}{h_{2}\left(i+1\right)}\right)$
are canceled.
\item Finally, if $\deg\left(h_{2}\right)>\deg\left(h_{1}\right)$, we expect
to see all sorts of factorials appearing in the denominators of the
summands, suggesting to look for Taylor expansions find get the limit.
\end{itemize}
\medskip{}
\end{rem}

In the examples above we only looked at ``trivial'' solutions where
$f\equiv1$. In general, there are solutions where $f\neq1$, however,
they are all part of the trivial family in disguise. Indeed, if
\begin{align*}
b\left(x\right) & =-h_{1}\left(x\right)h_{2}\left(x\right)\\
f\left(x\right)a\left(x\right) & =f\left(x-1\right)h_{1}\left(x\right)+f\left(x+1\right)h_{2}\left(x+1\right)
\end{align*}
as in the theorem, then using the equivalence transformation from
\lemref{equivalence-transformation} with $c_{n}=f\left(n\right)$,
we get that 
\[
\KK_{1}^{n}\frac{b\left(n\right)}{a\left(n\right)}=\frac{1}{f\left(0\right)}\KK_{1}^{n}\frac{\tilde{b}\left(n\right)}{\tilde{a}\left(n\right)}.
\]
where this new continued fraction is in the trivial Euler family:
\begin{align*}
\tilde{h}_{1}\left(x\right) & =h_{1}\left(x\right)f\left(x-1\right)\qquad;\qquad\tilde{h}_{2}\left(x\right)=h_{2}\left(x\right)f\left(x\right)\\
\tilde{b}\left(x\right) & =f\left(x-1\right)f\left(x\right)b\left(x\right)=-\tilde{h}_{1}\left(x\right)\times\tilde{h}_{2}\left(x\right)\\
\tilde{a}\left(x\right) & =f\left(x\right)a\left(x\right)=\tilde{h}_{1}\left(x\right)+\tilde{h}_{2}\left(x+1\right).
\end{align*}

\begin{rem}
It is also interesting to note that the same polynomial continued
fraction expansion $\KK_{1}^{\infty}\frac{b\left(i\right)}{a\left(i\right)}$
can have several presentations as an Euler continued fraction. For
example, given any two polynomials $g_{1}\left(x\right),g_{2}\left(x\right)$
and some constant $k$, suppose that we have
\[
b\left(x\right)=-g_{1}\left(x\right)g_{2}\left(x\right)\left(x+k\right)\left(x+k+1\right).
\]
On the one hand, we can decompose $-b\left(x\right)$ to 
\[
h_{1}\left(x\right)=g_{1}\left(x\right)\left(x+k\right)\;,\;h_{2}\left(x\right)=g_{2}\left(x\right)\left(x+k+1\right)
\]
and taking $f\left(x\right)\equiv1$ we obtain that 
\[
a\left(x\right)=h_{1}\left(x\right)+h_{2}\left(x+1\right).
\]
On the other hand, we can take the decomposition
\[
\tilde{h}_{1}\left(x\right)=g_{1}\left(x\right)\left(x+k+1\right)\;,\;\tilde{h}_{2}\left(x\right)=g_{2}\left(x\right)\left(x+k\right),
\]
and taking $f\left(x\right)=x+k+1$ we get the corresponding $\tilde{a}\left(x\right)$
as
\begin{align*}
\tilde{a}\left(x\right) & =\frac{\tilde{h}_{1}\left(x\right)f\left(x-1\right)+\tilde{h}_{2}\left(x+1\right)f\left(x+1\right)}{f\left(x\right)}\\
 & =\frac{\left[h_{1}\left(x\right)\frac{x+k+1}{x+k}\right]\cdot\left(x+k\right)+\left[h_{2}\left(x+1\right)\frac{x+k+1}{x+k+2}\right]\cdot\left(x+k+2\right)}{x+k+1}\\
 & =h_{1}\left(x\right)+h_{2}\left(x+1\right)=a\left(x\right).
\end{align*}
\end{rem}

\medskip{}

As mentioned before, the structure of the trivial family should not
be too surprising. Indeed, in a sense we are trying to find solutions
$\lambda_{1}\left(x\right),\lambda_{2}\left(x\right)$ to $S\left(t\right)=t^{2}+a\left(x\right)t+b\left(x\right)=0$,
as we have done in the constant case in \ref{subsec:Constant-continued-fraction}.
However, as in our continued fraction expansion we add $1$ to $x$
at each step, we instead look for solutions where we ``half advance''
the index $a\left(x\right)=\lambda_{1}\left(x\right)+\lambda_{2}\left(x+1\right)$.

In general, starting with a standard quadratic equation $x^{2}+ax+b=0$
over the integers, we do not expect to have solutions over the integers.
The same applies here - if we are given polynomials $a\left(x\right),b\left(x\right)$
where $\KK_{1}^{\infty}\frac{b\left(n\right)}{a\left(n\right)}$ converges,
it is not true that there is a solution as in \thmref{recurrence-roots}
with integral polynomials. Assuming that we know how to decompose
$b\left(x\right)$, we can go over all its decompositions $b\left(x\right)=-h_{1}\left(x\right)h_{2}\left(x\right)$
and see if $a\left(x\right)=h_{1}\left(x\right)+h_{2}\left(x+1\right)$.
In the more generalized form, where $f\left(x\right)a\left(x\right)=f\left(x-1\right)h_{1}\left(x\right)+f\left(x+1\right)h_{2}\left(x+1\right)$,
this process is less trivial, but still not that hard. As this algorithm
is much more technical in nature we describe it in \appref{Identifying-polynomial-continued},
and we also have an actual python code to find such solutions in \cite{ramanujan_machine_research_group_ramanujan_2023}.

Finally we investigate a little further into the special case where
$\deg\left(h_{1}\right),\deg\left(h_{2}\right)\leq1$ in \appref{Beta-and-Gamma}.
It is well known that polynomial continued fraction can represent
values of hypergeometric functions (see for example section 6 in \cite{jones_continued_1980}),
and are closely related to many interesting integrals. In \appref{Beta-and-Gamma}
we go over some of the details for this degree 1 case, and in particular
see how the Beta and Gamma functions appear naturally in this area.
\begin{problem}
One of the interesting questions raised once we have the presentation
of polynomial continued fraction in the Euler family, is can we use
this presentation to ``collect'' many polynomial continued fractions
into one family. For example, fixing the numerator polynomial $b\left(x\right)$,
are there infinitely many denominators $a_{k}\left(x\right),\;k\in\NN$
and corresponding $f_{k}\left(x\right)$ for which $\KK_{1}^{\infty}\frac{b\left(i\right)}{a_{k}\left(i\right)}$
is in the Euler family? And if so, is there a deeper connection between
those continued fractions? It appears that in many cases this is true,
and we study this idea in a second paper where we define a structure
call \textbf{the conservative matrix field} which combines many polynomial
continued fractions into one interesting structure. In particular,
this structure can be used to reprove Ap\'ery's result that $\zeta\left(3\right)$
is irrational \cite{apery_irrationalite_1979,van_der_poorten_proof_1979}.
\end{problem}

\newpage{}

\section{\label{sec:The-most-general}The most general of generalized continued
fractions}

One of the main steps in understanding the Euler continued fractions
was using the equivalence transformation in \lemref{equivalence-transformation}.
This raises the broader question of whether there is a more general
``equivalence'' relation between continued fractions which can be
useful in our investigation. This step had a natural interpretation
using matrices as well, suggesting that maybe the right framework
is to simply study polynomial matrices, which is what we do in this
section.
\begin{defn}
Let $M\left(i\right)\in\GL_{2}\left(\CC\right)$ be a sequence of
$2\times2$ matrices. For $z\in\CC$ we will denote 
\[
\left[\prod_{1}^{\infty}M\left(i\right)\right]\left(z\right)=\limfi n\left[\prod_{1}^{n}M\left(i\right)\right]\left(z\right).
\]
\end{defn}

In the definition above, it is possible that the limit converges for
one $z$ and diverges for another. For example, if we take $M\left(i\right)=\left(\begin{smallmatrix}-1 & 0\\
0 & 1
\end{smallmatrix}\right)$, then $\left[\prod_{1}^{\infty}M\left(i\right)\right]\left(0\right)=0$
while $\left[\prod_{1}^{n}M\left(i\right)\right]\left(1\right)=\left(-1\right)^{n}$
doesn't converge. However, in many ``natural'' sequences we have
a very strong convergence behavior.
\begin{example}
\begin{enumerate}
\item If $M_{i}=\left(\begin{smallmatrix}1 & a_{i}\\
0 & 1
\end{smallmatrix}\right)$ are upper triangular, then 
\[
\left(\prod_{1}^{n}M_{i}\right)\left(z\right)=\left(\begin{smallmatrix}1 & \sum_{1}^{n}a_{i}\\
0 & 1
\end{smallmatrix}\right)\left(z\right)=z+\sum_{1}^{n}a_{i}.
\]
If $\sum_{1}^{\infty}a_{i}=\infty$, then $\left(\prod_{1}^{n}M_{i}\right)\left(z\right)\to\infty$
for all $z$. If we also add an $M_{0}$ matrix which takes $\infty\to w\in\CC$,
then $\left(\prod_{0}^{n}M_{i}\right)\left(z\right)\to w$ for all
$z$.
\item If $M_{i}=\left(\begin{smallmatrix}0 & b_{i}\\
1 & a_{i}
\end{smallmatrix}\right)$ has the continued fraction form, then $\left(\prod_{1}^{n}M_{i}\right)\left(\infty\right)=\left(\prod_{1}^{n-1}M_{i}\right)\left(0\right)$
since $M_{i}\left(\infty\right)=0$. It follows that we have convergence
with the same limit at $0$ if and only if we have convergence at
$\infty$. Writing $\prod_{1}^{n-1}M_{i}=\left(\begin{smallmatrix}p_{n-1} & p_{n}\\
q_{n-1} & q_{n}
\end{smallmatrix}\right)$ , this limit will simply be $\limfi n\frac{p_{n}}{q_{n}}$. \\
Considering convergence for general starting point $x\in\RR$ we note
that
\[
\left|\left[\prod_{1}^{n-1}M_{i}\right]\left(x\right)-\left[\prod_{1}^{n-1}M_{i}\right]\left(0\right)\right|=\left|\frac{p_{n-1}x+p_{n}}{q_{n-1}x+q_{n}}-\frac{p_{n}}{q_{n}}\right|=\left|\frac{p_{n-1}}{q_{n-1}}-\frac{p_{n}}{q_{n}}\right|\left|\frac{x}{\left(x+\frac{q_{n}}{q_{n-1}}\right)}\right|.
\]
Since $\left|\frac{p_{n-1}}{q_{n-1}}-\frac{p_{n}}{q_{n}}\right|\to0$,
as long as $\left|\frac{x}{\left(x+\frac{q_{n}}{q_{n-1}}\right)}\right|$
is bounded, we will get that 
\[
\limfi n\left[\prod_{1}^{n-1}M_{i}\right]\left(x\right)=\limfi n\left[\prod_{1}^{n-1}M_{i}\right]\left(0\right)=\KK_{1}^{\infty}\frac{b_{i}}{a_{i}}.
\]
When $b_{i},a_{i}$ are constant, as in \subsecref{Constant-continued-fraction},
we expect $q_{n}\sim\lambda^{n}$ to grow exponentially, so that there
is convergence for any $x\neq-\lambda$. When $a_{n},b_{n}$ are non
constant polynomials, then usually $q_{n}$ grows at least factorially,
so that $\left|\frac{q_{n}}{q_{n-1}}\right|\to\infty$, and then $\left|\frac{x}{\left(x+\frac{q_{n}}{q_{n-1}}\right)}\right|$
is bounded as well. Moreover, this convergence is going to be uniform
in $x$ for $x$ in a bounded region (or more generally when $\left|1+\frac{q_{n}}{q_{n-1}x}\right|\geq C>0$
is bounded from below).
\end{enumerate}
\end{example}

\newpage{}

\subsection{\label{subsec:coboundary}Some words about cocycles and coboundaries}
\begin{flushleft}
Next, we would like to generalize the ``equivalence'' from \lemref{equivalence-transformation}.

In our previous discussion about continued fractions, where $M_{i}=\left(\begin{smallmatrix}0 & b_{i}\\
1 & a_{i}
\end{smallmatrix}\right)$, we were specially interested in $P_{n}:=\prod_{1}^{n}M_{i}=\left(\begin{smallmatrix}p_{n} & p_{n+1}\\
q_{n} & q_{n+1}
\end{smallmatrix}\right)$ which contained the numerators and denominators of the convergents.
These products can be defined for any sequence $M_{i}$ of matrices,
and we think of them as \textbf{potential matrices} where we move
from the potential at point $i$ to the potential at point $i+1$
via the matrix $M_{i}$, or more formally $P_{i}M_{i}=P_{i+1}$. We
can also restrict them to row vectors, instead of full $2\times2$
matrices. In particular, for the continued fraction matrices, each
such vector sequence will satisfy $\left(F_{i-1},F_{i}\right)M_{i}=\left(F_{i},F_{i+1}\right)$,
where $F_{i}$ solves the recurrence relation 
\[
F_{i+1}=F_{i}a_{i}+F_{i-1}b_{i}
\]
that we already encountered before in \subsecref{Euler's-formula}.

Attempting to move from the potential of one sequence to another is
quite natural, and this type of question is usually asked in cohomology
theory (see for example chapter 4 in \cite{brown_cohomology_2012}),
where such sequences $M_{i}$ of matrices can and should be called
``\textbf{cocycles}''. We will not go too much into this theory
here, since on the one hand this cocycle structure is in a sense trivial
here, and on the other hand, the theory is usually much more geared
into commutative rings, unlike our noncommutative matrix setting.
However, the question about natural conversion between the potentials
exists in this theory and is called \textbf{coboundary equivalence},
and this will come up a lot in our study. More specifically we want
to move from one potential $P_{i}^{\left(1\right)}$ to the second
$P_{i}^{\left(2\right)}$ using a nice transformation $P_{i}^{\left(1\right)}U_{i}=P_{i}^{\left(2\right)}$,
producing for us this commutative diagram:
\[
\xymatrix{P_{1}^{\left(1\right)}\ar[r]^{M_{1}^{\left(1\right)}}\ar[d]^{U_{1}} & P_{2}^{\left(1\right)}\ar[r]^{M_{2}^{\left(1\right)}}\ar[d]^{U_{2}} & P_{3}^{\left(1\right)}\ar[r]^{M_{3}^{\left(1\right)}}\ar[d]^{U_{3}} & \cdots\ar[r]^{M_{n-1}^{\left(1\right)}} & P_{n}^{\left(1\right)}\ar[d]^{U_{n}}\\
P_{1}^{\left(2\right)}\ar[r]^{M_{1}^{\left(2\right)}} & P_{2}^{\left(2\right)}\ar[r]^{M_{2}^{\left(2\right)}} & P_{3}^{\left(2\right)}\ar[r]^{M_{3}^{\left(2\right)}} & \cdots\ar[r]^{M_{n-1}^{\left(2\right)}} & P_{n}^{\left(2\right)}
}
\]
More formally, we have the following.
\begin{defn}
Two matrix sequences $M_{i}^{\left(1\right)},\;M_{i}^{\left(2\right)}$
are called \textbf{$U_{i}$-coboundary equivalent} for a sequence
of invertible matrices $U_{i}$ if $M_{i}^{\left(1\right)}U_{i+1}=U_{i}M_{i}^{\left(2\right)}$
for all $i$, or in a commutative diagram form:
\[
\xymatrix{\left(*\right)\ar[r]^{M_{i}^{\left(2\right)}} & \left(*\right)\\
\left(*\right)\ar[r]_{M_{i}^{\left(1\right)}}\ar[u]^{U_{i}} & \left(*\right)\ar[u]_{U_{i+1}}
}
.
\]
\end{defn}

Note that in general we work with Mobius transformation, so that $M_{i}^{\left(1\right)}U_{i+1}=U_{i}M_{i}^{\left(2\right)}$
should be true up to a scalar matrix multiplication.
\begin{rem}
In the world of standard, non indexed matrices, this coboundary equivalence
is simply matrix conjugation, and as we shall see some of the results
for this coboundary equivalence are just ``indexed'' versions of
what we expect from matrix conjugacy.
\end{rem}

The commutativity condition in the coboundary definition can be extended
to products of the $M_{i}$, and in particular for our potential matrices.
Indeed, a simple induction shows that with the notations as in the
definition, for all $m\leq n$ we have
\[
U_{m}\left[\prod_{m}^{n}M_{i}^{\left(2\right)}\right]=\left[\prod_{m}^{n}M_{i}^{\left(1\right)}\right]U_{n+1}.
\]

Every two (invertible) matrix sequences $M_{i}^{\left(1\right)},\;M_{i}^{\left(2\right)}$
are coboundary equivalent for some $U_{i}$. Indeed, once we choose
$U_{1}$, we can recursively define $U_{i+1}=\left(M_{i}^{\left(2\right)}\right)^{-1}U_{i}M_{i}^{\left(1\right)}$.
However, what will matter to us later on is that $U_{i}$ is simple
enough to work with. For example, it can be defined over $\ZZ$, triangular,
diagonal, etc. In particular, we want to work with the Mobius maps
induced by the matrices, and the upper triangular (resp. lower triangular)
are exactly the matrices which take infinity to itself (resp. zero
to itself).

This idea of coboundary equivalent sequences is very useful, and the
``equivalence transformation'' from \lemref{equivalence-transformation}
was just one example: 
\[
M_{n}^{\left(1\right)}=\left(\begin{smallmatrix}0 & b_{n}\\
1 & a_{n}
\end{smallmatrix}\right),\quad M_{n}^{\left(2\right)}=\left(\begin{smallmatrix}0 & c_{n-1}c_{n}b_{n}\\
1 & c_{n}a_{n}
\end{smallmatrix}\right),\quad U_{n}=\left(\begin{smallmatrix}1 & 0\\
0 & c_{n-1}
\end{smallmatrix}\right)
\]
so that $M_{n}^{\left(1\right)}U_{n+1}=c_{n-1}U_{n}M_{n}^{\left(2\right)}$.
Since we deal with Mobius transformation, where scalar matrices act
as the identity, we have $M_{n}^{\left(1\right)}U_{n+1}\equiv U_{n}M_{n}^{\left(2\right)}$.

\medskip{}

Other than this important example, we have two more - one to move
to upper triangular matrices, and one to continued fraction matrices,
both of which are helpful when we need to take product of many such
matrices. In the upper triangular case, the diagonal is just a product
of the diagonals and the only complicated part is in the upper right
corner. In the continued fraction form, as we already saw, there is
a natural recursion relation, which will be very helpful once we start
to do the actual computations, and we will start with it. This conversion
was first stated in \cite{razon_automated_2022}.
\begin{thm}[\textbf{\textit{Conversion to continued fraction}}]
\label{thm:to-gcf} Let 
\[
M_{n}=\left(\begin{smallmatrix}a_{n} & b_{n}\\
c_{n} & d_{n}
\end{smallmatrix}\right),\quad U_{n}=\left(\begin{smallmatrix}1 & a_{n}\\
0 & c_{n}
\end{smallmatrix}\right),\quad a_{n},b_{n},c_{n},d_{n}\in\CC,
\]
such that $c_{n}\neq0$ for all $n\geq1$ (so that $U_{n}$ is invertible).
Then $M_{n}$ is $U_{n}$-coboundary equivalent to the continued fraction
matrix
\[
U_{n}^{-1}M_{n}U_{n+1}=\left(\begin{smallmatrix}0 & -\frac{c_{n+1}}{c_{n}}\det\left(M_{n}\right)\\
1 & a_{n+1}+d_{n}\frac{c_{n+1}}{c_{n}}
\end{smallmatrix}\right).
\]
In particular, setting $\left(\begin{smallmatrix}p_{n}\\
q_{n}
\end{smallmatrix}\right)=M_{1}M_{2}\cdots M_{n}\left(\begin{smallmatrix}1\\
0
\end{smallmatrix}\right)$, both of the $p_{n}$ and $q_{n}$ satisfy the same recurrence relation
\[
\left(\begin{smallmatrix}p_{n} & p_{n+1}\\
q_{n} & q_{n+1}
\end{smallmatrix}\right)=\left(\begin{smallmatrix}p_{n-1} & p_{n}\\
q_{n-1} & q_{n}
\end{smallmatrix}\right)\left(U_{n}^{-1}M_{n}U_{n+1}\right)\quad;\quad\left(\begin{smallmatrix}p_{0} & p_{1}\\
q_{0} & q_{1}
\end{smallmatrix}\right)=\left(\begin{smallmatrix}1 & a_{1}\\
0 & c_{1}
\end{smallmatrix}\right).
\]
\end{thm}

\begin{proof}
The computation of $\tilde{M}_{n}:=U_{n}^{-1}M_{n}U_{n+1}=\left(\begin{smallmatrix}0 & -\frac{c_{n+1}}{c_{n}}\det\left(M_{n}\right)\\
1 & a_{n+1}+d_{n}\frac{c_{n+1}}{c_{n}}
\end{smallmatrix}\right)$ is straight forward. Noting that $U_{n+1}\left(\begin{smallmatrix}1\\
0
\end{smallmatrix}\right)=\left(\begin{smallmatrix}1\\
0
\end{smallmatrix}\right)$, the coboundary equivalence gives us
\begin{align*}
\left(\begin{smallmatrix}p_{n}\\
q_{n}
\end{smallmatrix}\right) & =\left(\prod_{1}^{n}M_{k}\right)U_{n+1}\left(\begin{smallmatrix}1\\
0
\end{smallmatrix}\right)=U_{1}\left(\prod_{1}^{n}\tilde{M}_{k}\right)\left(\begin{smallmatrix}1\\
0
\end{smallmatrix}\right).
\end{align*}
Since $\tilde{M}_{n+1}\left(\begin{smallmatrix}1\\
0
\end{smallmatrix}\right)=\left(\begin{smallmatrix}0\\
1
\end{smallmatrix}\right)$, we also get that 
\[
\left[U_{1}\prod_{1}^{n}\tilde{M}_{k}\right]\left(\begin{smallmatrix}0\\
1
\end{smallmatrix}\right)=\left[U_{1}\prod_{1}^{n}\tilde{M}_{k}\right]\tilde{M}_{n+1}\left(\begin{smallmatrix}1\\
0
\end{smallmatrix}\right)=\left(\begin{smallmatrix}p_{n+1}\\
q_{n+1}
\end{smallmatrix}\right).
\]
In other words, we got that 
\[
U_{1}\prod_{1}^{n}\tilde{M}_{k}=\left(\begin{smallmatrix}p_{n} & p_{n+1}\\
q_{n} & q_{n+1}
\end{smallmatrix}\right).
\]
This implies the matrix recurrence relation 
\[
\left(\begin{smallmatrix}p_{n-1} & p_{n}\\
q_{n-1} & q_{n}
\end{smallmatrix}\right)\tilde{M}_{n}=\left(\begin{smallmatrix}p_{n} & p_{n+1}\\
q_{n} & q_{n+1}
\end{smallmatrix}\right)\quad;\quad\left(\begin{smallmatrix}p_{0} & p_{1}\\
q_{0} & q_{1}
\end{smallmatrix}\right)=U_{1}=\left(\begin{smallmatrix}1 & a_{1}\\
0 & c_{1}
\end{smallmatrix}\right).
\]
\end{proof}
\begin{rem}
In the theorem above, if the entries of $M_{n}$ are integers with
nonconstant $c_{n}$, then in general the coefficients $U_{n}^{-1}M_{n}U_{n+1}$
in the recurrence are nonintegers. However, the $p_{n},q_{n}$ solutions
will still be integers, since we used the product of the integral
$M_{i}$ matrices to define them. If we only care about the ratio
$\frac{p_{n}}{q_{n}}$, then we may use the original equivalence transformation
from \lemref{equivalence-transformation} to move to the integral
coefficients
\[
\left(\begin{smallmatrix}c_{n-1} & 0\\
0 & 1
\end{smallmatrix}\right)\left(\begin{smallmatrix}0 & -\frac{c_{n+1}}{c_{n}}\det\left(M_{n}\right)\\
1 & a_{n+1}+d_{n}\frac{c_{n+1}}{c_{n}}
\end{smallmatrix}\right)\left(\begin{smallmatrix}1 & 0\\
0 & c_{n}
\end{smallmatrix}\right)=\left(\begin{smallmatrix}0 & -c_{n-1}c_{n+1}\det\left(M_{n}\right)\\
1 & c_{n}a_{n+1}+d_{n}c_{n+1}
\end{smallmatrix}\right).
\]
\end{rem}

We are mainly interested in the case where the matrices' entries are
polynomial evaluated at the integer points, so for example in the
previous case $c_{n}=c\left(n\right)$ where $c\in\CC\left[x\right]$.
In particular, unless $c\equiv0$, in which case the $M_{n}$ are
upper triangular, we can always apply this transformation for all
$n$ large enough. Moreover, if $M_{n}$ is a polynomial matrix with
polynomials over $\ZZ$, and $c_{n}$ are constant, then its polynomial
continued fraction form above is also over $\ZZ$, without using the
equivalence transformation.

To summerize, the theorem above shows that almost any polynomial matrix
is ``nicely'' coboundary equivalent to a polynomial continued fraction,
and if it is not, then it is upper triangular, which is usually easier
to work with, since we can use equation like $\prod_{1}^{n}\left(\begin{smallmatrix}1 & \alpha_{i}\\
0 & 1
\end{smallmatrix}\right)=\left(\begin{smallmatrix}1 & \sum_{1}^{n}\alpha_{i}\\
0 & 1
\end{smallmatrix}\right)$.
\par\end{flushleft}

\subsection{Upper triangular matrices}

Our next goal is to show when we can transform a sequence of matrices
into upper triangular. Recall that a standard $2\times2$ matrix is
conjugated to an upper triangular matrix if and only if it has a nonzero
eigenvector $v^{tr}M=\lambda v^{tr}$. Here we also have the indexed
analogue, which while at first glance seems a bit trivial, when we
add the right restrictions, will become quite helpful.
\begin{defn}
Let $M_{i}=\left(\begin{smallmatrix}a_{i} & b_{i}\\
c_{i} & d_{i}
\end{smallmatrix}\right)$ be a sequence of matrices. We say that a sequence $v\left(i\right)=\left(F_{i},G_{i}\right)$
of nonzero vectors is a \textbf{left eigenvector} with eigenvalue
$\lambda\left(i\right)$ if 
\[
v\left(i\right)M_{i}=\lambda\left(i\right)v\left(i+1\right).
\]

We similarly define \textbf{right eigenvector} and right eigenvalue
by the formula
\[
M_{i}u\left(i+1\right)=\alpha\left(i\right)u\left(i\right).
\]
\end{defn}

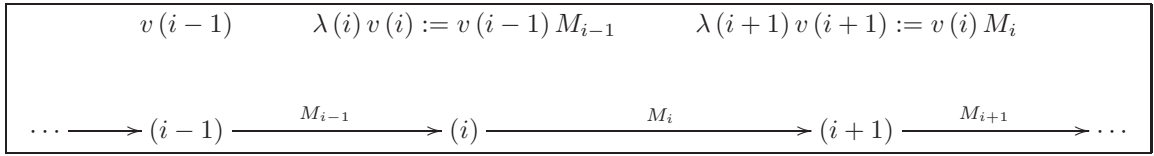
\begin{figure}[H]
\begin{centering}
\begin{tabular}{|c|}
\hline 
$\xymatrix{ & v\left(i-1\right) & \lambda\left(i\right)v\left(i\right):=v\left(i-1\right)M_{i-1} & \lambda\left(i+1\right)v\left(i+1\right):=v\left(i\right)M_{i}\\
\cdots\ar[r] & \left(i-1\right)\ar[r]^{M_{i-1}} & \left(i\right)\ar[r]^{M_{i}} & \left(i+1\right)\ar[r]^{M_{i+1}} & \cdots
}
$\tabularnewline
\hline 
\end{tabular}
\par\end{centering}
\caption{We should think of left eigenvectors $v\left(i\right)$ as being at
the $i$-th position, and multiplying by $M_{i}$ from the right \textquotedblleft moves\textquotedblright{}
them to the $i+1$ position. Right eigenvector goes similarly but
from $i+1$ to $i$ with $M_{i}$ multiplying from the left.}
\end{figure}

Unlike standard eigenvectors, in our case it is very easy to find
eigenvectors by simply defining $v\left(i+1\right)=\frac{1}{\lambda\left(i\right)}v\left(i\right)M\left(i\right)$
recursively. However, the problem is finding an eigenvector which
is easy to work with. We already saw one such example when our $M_{i}$
had continued fraction form, in which case a 1-left eigenvector is
simply a solution to 
\[
\left(F_{i-1},F_{i}\right)\left(\begin{smallmatrix}0 & b_{i}\\
1 & a_{i}
\end{smallmatrix}\right)=\left(F_{i},F_{i+1}\right),
\]
or alternatively $F_{i+1}=a_{i}F_{i}+b_{i}F_{i-1}$. If $F_{i}$ is
just any sequence, then it would be very hard to work with, however
if $b_{i},a_{i}$ are polynomial in $i$, we might find $F_{i}$ which
is polynomial or exponential in $i$.
\begin{example}
\label{exa:euler-left-right-eigenvectors}Consider a generalized continued
fraction from the trivial Euler family $M_{i}=\left(\begin{smallmatrix}0 & -h_{1}\left(i\right)h_{2}\left(i\right)\\
1 & h_{1}\left(i\right)+h_{2}\left(i+1\right)
\end{smallmatrix}\right)$ (see \subsecref{Euler's-formula}). Then it has both a $h_{2}\left(i\right)$-left
and $h_{1}\left(i\right)$-right eigenvectors
\begin{align*}
\left(1,\;h_{2}\left(i\right)\right)\cdot\left(\begin{smallmatrix}0 & -h_{1}\left(i\right)h_{2}\left(i\right)\\
1 & h_{1}\left(i\right)+h_{2}\left(i+1\right)
\end{smallmatrix}\right) & =h_{2}\left(i\right)\cdot\left(1,\;h_{2}\left(i+1\right)\right)\\
\left(\begin{smallmatrix}0 & -h_{1}\left(i\right)h_{2}\left(i\right)\\
1 & h_{1}\left(i\right)+h_{2}\left(i+1\right)
\end{smallmatrix}\right)\cdot\left(\begin{smallmatrix}h_{2}\left(i+1\right)\\
-1
\end{smallmatrix}\right) & =h_{1}\left(i\right)\cdot\left(\begin{smallmatrix}h_{2}\left(i\right)\\
-1
\end{smallmatrix}\right).
\end{align*}
In particular if $h_{1},h_{2}\in\ZZ\left[x\right]$, then both eigenvalues
and eigenvectors are integral.
\end{example}

\medskip{}

Note that in the example above, the left and right eigenvectors at
the $i$-position are perpendicular:
\[
\left(1,\;h_{2}\left(i\right)\right)\cdot\left(\begin{smallmatrix}h_{2}\left(i\right)\\
-1
\end{smallmatrix}\right)=0.
\]
This is not a coincidence, and it happens for any sequence of matrices,
and we can use it to triangularize it.
\begin{lem}
\label{lem:triangularization}Let $M_{i}=\left(\begin{smallmatrix}a_{i} & b_{i}\\
c_{i} & d_{i}
\end{smallmatrix}\right)$ be any sequence of matrices. Then $\left(G_{i},F_{i}\right)$ is
a left $\lambda_{i}$-eigenvector if and only if $\left(\begin{smallmatrix}F_{i}\\
-G_{i}
\end{smallmatrix}\right)$ is a right $\alpha_{i}$-eigenvector. Moreover, if the $F_{i}\neq0$,
then setting $U_{n}=\left(\begin{smallmatrix}F_{i}^{-1} & 0\\
G_{i} & F_{i}
\end{smallmatrix}\right)$ we get that 
\[
U_{i}M_{i}U_{i+1}^{-1}=\left(\begin{smallmatrix}\alpha_{i} & \frac{b_{i}}{F_{i}F_{i+1}}\\
0 & \lambda_{i}
\end{smallmatrix}\right),
\]
so in particular $\alpha_{i}\lambda_{i}=\det\left(M_{i}\right)$.
\end{lem}

\begin{proof}
Suppose that $M_{i}$ has a $\lambda_{i}$-left eigenvector $\left(G_{i},F_{i}\right)$.
Then 
\[
\left(G_{i},F_{i}\right)M_{i}\left(\begin{smallmatrix}F_{i+1}\\
-G_{i+1}
\end{smallmatrix}\right)=\lambda_{i}\left(G_{i+1},F_{i+1}\right)\left(\begin{smallmatrix}F_{i+1}\\
-G_{i+1}
\end{smallmatrix}\right)=0,
\]
implying that $\text{\ensuremath{M_{i}\left(\begin{smallmatrix}F_{i+1}\\
 -G_{i+1} 
\end{smallmatrix}\right)}\ensuremath{\ensuremath{\perp\left(G_{i},F_{i}\right)}}}$. In our 2-dimensional space, the perpendicular space of a nonzero
vector is a 1-dimensional space, so that $M_{i}\left(\begin{smallmatrix}F_{i+1}\\
-G_{i+1}
\end{smallmatrix}\right)=\alpha_{i}\left(\begin{smallmatrix}F_{i}\\
-G_{i}
\end{smallmatrix}\right)$ for some scalar $\alpha_{i}$, namely it is a right eigenvector.

For our choice of $U_{i}$ the row $e_{2}^{tr}U_{i}$ is the left
eigenvector, and since $U_{i+1}^{-1}=\left(\begin{smallmatrix}F_{i+1} & 0\\
-G_{i+1} & F_{i+1}^{-1}
\end{smallmatrix}\right)$ the column $U_{i+1}^{-1}e_{1}$ is the right eigenvector. We now
get that 
\begin{align*}
e_{2}^{tr}U_{i}M_{i}U_{i+1}^{-1} & =\lambda_{i}e_{2}^{tr}U_{i+1}U_{i+1}^{-1}=\lambda_{i}e_{2}^{tr}\\
U_{i}M_{i}U_{i+1}^{-1}e_{1} & =\alpha_{i}U_{i}U_{i}^{-1}e_{1}=\alpha_{i}e_{1}\\
e_{1}^{tr}U_{i}M_{i}U_{i+1}^{-1}e_{2} & =\left(F_{i}^{-1}e_{1}^{tr}\right)M_{i}\left(F_{i+1}^{-1}e_{2}\right)=\frac{b_{i}}{F_{i}F_{i+1}}.
\end{align*}
Hence, all together we get that 
\[
U_{i}M_{i}U_{i+1}^{-1}=\left(\begin{smallmatrix}\alpha_{i} & \frac{b_{i}}{F_{i}F_{i+1}}\\
0 & \lambda_{i}
\end{smallmatrix}\right),
\]
which completes the proof.
\end{proof}
\newpage{}

To fully utilize this triangularization let's recall the formula for
computing a product of triangular matrices.
\begin{claim}
\label{claim:upper-triangular}For sequences $\alpha_{i},\beta_{i},\gamma_{i}$
with $\alpha_{i},\gamma_{i}\neq0$ we have that 
\begin{align*}
\prod_{1}^{n-1}\left(\begin{smallmatrix}\alpha_{i} & \beta_{i}\\
0 & \gamma_{i}
\end{smallmatrix}\right) & =\left(\begin{smallmatrix}\prod_{1}^{n-1}\alpha_{i} & c_{n}\\
0 & \prod_{1}^{n-1}\gamma_{i}
\end{smallmatrix}\right)\\
c_{n} & =\sum_{k=1}^{n-1}\left(\prod_{i=1}^{k-1}\alpha_{i}\right)\beta_{k}\left(\prod_{i=k+1}^{n-1}\gamma_{i}\right)
\end{align*}
In particular, as a Mobius map we get 
\[
\left[\prod_{1}^{n-1}\left(\begin{smallmatrix}\alpha_{i} & \beta_{i}\\
0 & \gamma_{i}
\end{smallmatrix}\right)\right]\left(0\right)=\sum_{k=1}^{n-1}\frac{\beta_{k}}{\gamma_{k}}\left(\prod_{i=1}^{k-1}\frac{\alpha_{i}}{\gamma_{i}}\right)
\]
.
\end{claim}

\begin{proof}
A simple induction which is left as an exercise.
\end{proof}
\begin{example}
\label{exa:Reproving_euler_conversion}The last two results can be combined together to reprove the conversion
from continued fractions in the Euler family from \exaref{Euler-family}
to infinite sum. Given two functions $h_{1},h_{2}:\ZZ\to\CC$ set
\begin{align*}
b_{i} & =-h_{1}\left(i\right)h_{2}\left(i\right)\\
a_{i} & =h_{1}\left(i\right)+h_{2}\left(i+1\right),
\end{align*}
and let $M_{i}=\left(\begin{smallmatrix}0 & b_{i}\\
1 & a_{i}
\end{smallmatrix}\right)$. We already saw in \exaref{euler-left-right-eigenvectors} that this
sequence has $h_{2}\left(i\right)$-left eigenvector $\left(1,h_{2}\left(i\right)\right)$,
so \lemref{triangularization} implies that for $U_{i}=\left(\begin{smallmatrix}h_{2}\left(i\right)^{-1} & 0\\
1 & h_{2}\left(i\right)
\end{smallmatrix}\right)$ we have
\[
U_{i}M_{i}U_{i+1}^{-1}=\left(\begin{smallmatrix}h_{1}\left(i\right) & -\frac{h_{1}\left(i\right)}{h_{2}\left(i+1\right)}\\
0 & h_{2}\left(i\right)
\end{smallmatrix}\right).
\]
Then \claimref{upper-triangular} shows that 
\begin{align*}
\left(U_{1}\left[\prod_{1}^{n-1}M_{i}\right]U_{n+1}^{-1}\right)\left(0\right) & =-\frac{1}{h_{2}\left(1\right)}\cdot\sum_{k=1}^{n-1}\left(\prod_{i=1}^{k}\frac{h_{1}\left(i\right)}{h_{2}\left(i+1\right)}\right).
\end{align*}
Set $\alpha={\displaystyle \sum_{k=0}^{n-1}}\left({\displaystyle \prod_{i=1}^{k}}\frac{h_{1}\left(i\right)}{h_{2}\left(i+1\right)}\right)$,
so the expression above is $\frac{1-\alpha}{h_{2}\left(1\right)}$.
Since $U_{n+1}^{-1}\left(0\right)=0$, we conclude that 
\begin{align*}
\KK_{1}^{n-1}\frac{b_{i}}{a_{i}} & =\left[\prod_{1}^{n-1}M_{i}\right]\left(0\right)=U_{1}^{-1}\left(\frac{1-\alpha}{h_{2}\left(1\right)}\right)=\left(\begin{smallmatrix}h_{2}\left(1\right) & 0\\
-1 & h_{2}\left(1\right)^{-1}
\end{smallmatrix}\right)\left(\frac{1-\alpha}{h_{2}\left(1\right)}\right)\\
 & =\frac{1-\alpha}{\frac{1}{h_{2}\left(1\right)}\left(\alpha-1\right)+\frac{1}{h_{2}\left(1\right)}}=h_{2}\left(1\right)\left(\alpha^{-1}-1\right),
\end{align*}
which is exactly what we got in \exaref{Euler-family}.
\end{example}

\newpage{}

\part{Appendix}

\appendix

\section{\label{app:Identifying-polynomial-continued}Identifying polynomial
continued fractions in the Euler family}

Recall from \thmref{recurrence-roots} that given polynomials

\begin{align*}
b\left(x\right) & =-h_{1}\left(x\right)h_{2}\left(x\right)\\
f\left(x\right)a\left(x\right) & =f\left(x-1\right)h_{1}\left(x\right)+f\left(x+1\right)h_{2}\left(x+1\right)
\end{align*}
we can compute the convergent of the continued fractions as 
\[
\KK_{1}^{n}\frac{b_{i}}{a_{i}}=\frac{f\left(1\right)h_{2}\left(1\right)}{f\left(0\right)}\left(\frac{1}{\sum_{k=0}^{n}\frac{f\left(0\right)f\left(1\right)}{f\left(k\right)f\left(k+1\right)}\prod_{i=1}^{k}\left(\frac{h_{1}\left(i\right)}{h_{2}\left(i+1\right)}\right)}-1\right).
\]

Given only $a\left(x\right)$ and $b\left(x\right)$, and all possible
decompositions of $b\left(x\right)$, we can easily find if there
is a solution to the forms above with $f\left(x\right)\equiv1$. When
$f$ can have general degree, this task is more complicated, and in
this section we describe an algorithm which solves it. A python version
solving this problem can be found in  \cite{ramanujan_machine_research_group_ramanujan_2023}.

With this in mind we have the following results, which can eventually
be used to construct an algorithm which finds $f\left(x\right)$ (if
such a polynomial exists).
\begin{lem}
\label{lem:polynomial-recursion}Suppose that there is a solution
for an equation of the form
\begin{equation}
f\left(x+1\right)\beta_{\left(1\right)}\left(x\right)+f\left(x\right)\beta_{\left(0\right)}\left(x\right)+f\left(x-1\right)\beta_{\left(-1\right)}\left(x\right)=0,\label{eq:pol-recursion-formula}
\end{equation}
where $f,\beta_{\left(i\right)}\in\CC\left[x\right]$ are nonzero
polynomials. Let $d_{f}=\deg\left(f\right)$, $d=\max\left\{ \deg\left(\beta_{\left(i\right)}\right)\;\mid\;i=-1,0,1\right\} $
and write 
\begin{align*}
\beta_{\left(i\right)}\left(x\right) & =\sum_{j=0}^{d}\beta_{\left(i\right)}^{\left(j\right)}x^{j}
\end{align*}
where the coefficients $\beta_{\left(i\right)}^{\left(j\right)}\in\CC$
are scalars (and we use the convention of $\beta_{\left(i\right)}^{\left(j\right)}=0$
for negative $j$). Then
\begin{enumerate}
\item The sum $\beta_{\left(-1\right)}^{\left(d\right)}+\beta_{\left(0\right)}^{\left(d\right)}+\beta_{\left(1\right)}^{\left(d\right)}=0$.
In particular, at least two of the $\beta_{\left(i\right)}$ have
the max degree $d$.
\item If $\beta_{\left(-1\right)}^{\left(d\right)}-\beta_{\left(1\right)}^{\left(d\right)}\neq0$,
then the degree of $f$ must be $d_{f}=\frac{\beta_{\left(-1\right)}^{\left(d-1\right)}+\beta_{\left(0\right)}^{\left(d-1\right)}+\beta_{\left(1\right)}^{\left(d-1\right)}}{\beta_{\left(-1\right)}^{\left(d\right)}-\beta_{\left(1\right)}^{\left(d\right)}}$.
In particular, this expression must be well defined and an integer.
\item If $\beta_{\left(-1\right)}^{\left(d\right)}-\beta_{\left(1\right)}^{\left(d\right)}=0$,
then $\beta_{\left(-1\right)}^{\left(d\right)}+\beta_{\left(1\right)}^{\left(d\right)}\neq0$
and
\[
\left(\beta_{\left(-1\right)}^{\left(d-2\right)}+\beta_{\left(0\right)}^{\left(d-2\right)}+\beta_{\left(1\right)}^{\left(d-2\right)}\right)+d_{f}\left(-\beta_{\left(-1\right)}^{\left(d-1\right)}+\beta_{\left(1\right)}^{\left(d-1\right)}\right)+\binom{d_{f}}{2}\left(\beta_{\left(-1\right)}^{\left(d\right)}+\beta_{\left(1\right)}^{\left(d\right)}\right)=0
\]
is a nontrivial quadratic equation in $d_{f}$.
\end{enumerate}
\end{lem}

\newpage{}
\begin{proof}
In general, the coefficients of a product of polynomials is a convolution
of the coefficients of the given polynomials. In order to use this,
we first want to find the coefficients of $f_{\left(k\right)}=f\left(x+k\right)$
for $k=-1,0,1$, so that 
\[
\sum_{k=-1}^{1}f_{\left(k\right)}\left(x\right)\beta_{\left(k\right)}\left(x\right)=0.
\]
Writing $f\left(x\right)=\sum_{0}^{d_{f}}f^{\left(i\right)}\cdot x^{i}$
where $f^{\left(i\right)}\in\CC$, we get that
\[
f_{\left(k\right)}\left(x\right)=\sum_{0}^{d_{f}}f^{\left(i\right)}\cdot\left(x+k\right)^{i}=\sum_{i=0}^{d_{f}}f^{\left(i\right)}\cdot\sum_{j=0}^{i}\binom{i}{j}x^{j}k^{i-j}=\sum_{j=0}^{d_{f}}x^{j}\sum_{i=j}^{d_{f}}f^{\left(i\right)}\cdot\binom{i}{j}k^{i-j}.
\]

For $i<j$, we can write $\binom{i}{j}=0$, so that the coefficient
of $x^{j}$ in $f_{\left(k\right)}$ is 
\[
f_{\left(k\right)}^{\left(j\right)}=\sum_{i=0}^{d_{f}}f^{\left(i\right)}\cdot\binom{i}{j}k^{i-j}.
\]
The coefficient of $x^{d_{f}+d-\ell}$ in $\sum_{k=-1}^{1}f_{\left(k\right)}\left(x\right)\cdot\beta_{\left(k\right)}\left(x\right)$
is
\begin{align*}
\sum_{k=-1}^{1}\left[\sum_{j=0}^{\ell}f_{\left(k\right)}^{\left(d_{f}-j\right)}\beta_{\left(k\right)}^{\left(d+j-\ell\right)}\right] & =\sum_{k=-1}^{1}\sum_{j=0}^{\ell}\sum_{i=0}^{d_{f}}f^{\left(i\right)}\cdot\binom{i}{d_{f}-j}k^{i+j-d_{f}}\beta_{\left(k\right)}^{\left(d+j-\ell\right)}\\
 & =\sum_{i=0}^{d_{f}}f^{\left(i\right)}\left[\sum_{j=d_{f}-i}^{\ell}\binom{i}{d_{f}-j}\sum_{k=-1}^{1}k^{i+j-d_{f}}\beta_{\left(k\right)}^{\left(d+j-\ell\right)}\right].
\end{align*}

\begin{enumerate}
\item We first look at the leading coefficient, namely the coefficient of
$x^{d_{f}+d}$, which should be zero. This means that $\ell=0$, implying
that from all the sums we are left with $j=0$ and $i=d_{f}$, so
that 
\[
0=f^{\left(d_{f}\right)}\left[\sum_{k=-1}^{1}\beta_{\left(k\right)}^{\left(d\right)}\right].
\]
The leading coefficient of $f$ is non zero, so we are left with
\[
0=\beta_{\left(-1\right)}^{\left(d\right)}+\beta_{\left(0\right)}^{\left(d\right)}+\beta_{\left(1\right)}^{\left(d\right)}.
\]
Since this sum is zero, and at least one of the summands is nonzero
(since $d=\max\left\{ \deg\left(\beta_{\left(k\right)}\right)\;\mid\;k=-1,0,1\right\} $),
at least two of them are non zero.\\
\newpage{}
\item Next, taking $\ell=1$ and equating the coefficient to 0, we get that
\begin{align*}
0 & =\sum_{i=0}^{d_{f}}f^{\left(i\right)}\left[\sum_{j=d_{f}-i}^{1}\binom{i}{d_{f}-j}\sum_{k=-1}^{1}k^{i+j-d_{f}}\beta_{\left(k\right)}^{\left(d+j-\ell\right)}\right]\\
 & =\overbrace{f^{\left(d_{f}\right)}\left[\sum_{j=0}^{1}\binom{d_{f}}{d_{f}-j}\sum_{k=-1}^{1}k^{j}\beta_{\left(k\right)}^{\left(d+j-1\right)}\right]}^{i=d_{f}}+\overbrace{f^{\left(d_{f}-1\right)}\left[\sum_{k=-1}^{1}\beta_{\left(k\right)}^{\left(d\right)}\right]}^{i=d_{f}-1}.
\end{align*}
From part (1) we know that $\left[\sum_{k=-1}^{1}\beta_{\left(k\right)}^{\left(d\right)}\right]=0$.
Using again the fact that $f^{\left(d_{f}\right)}\neq0$, we get that
\[
0=\sum_{k=-1}^{1}\beta_{\left(k\right)}^{\left(d-1\right)}+d_{f}\sum_{k=-1}^{1}k\beta_{\left(k\right)}^{\left(d\right)}.
\]
In particular, if $\beta_{\left(1\right)}^{\left(d\right)}-\beta_{\left(-1\right)}^{\left(d\right)}=\sum_{k=-1}^{1}k\beta_{\left(k\right)}^{\left(d\right)}\neq0$,
then 
\[
d_{f}=-\frac{\sum_{k=-1}^{1}\beta_{\left(k\right)}^{\left(d-1\right)}}{\sum_{k=-1}^{1}k\beta_{\left(k\right)}^{\left(d\right)}}.
\]
Otherwise, we get that $\sum_{k=-1}^{1}k\beta_{\left(k\right)}^{\left(d\right)}=\sum_{k=-1}^{1}\beta_{\left(k\right)}^{\left(d-1\right)}=0$.
\item Finally, letting $\ell=2$, we get
\begin{align*}
0 & =\sum_{i=0}^{d_{f}}f^{\left(i\right)}\left[\sum_{j=d_{f}-i}^{2}\binom{i}{d_{f}-j}\sum_{k=-1}^{1}k^{i+j-d_{f}}\beta_{\left(k\right)}^{\left(d+j-2\right)}\right]\\
 & =f^{\left(d_{f}\right)}\left[\sum_{j=0}^{2}\binom{d_{f}}{d_{f}-j}\sum_{k=-1}^{1}k^{j}\beta_{\left(k\right)}^{\left(d+j-2\right)}\right]+f^{\left(d_{f}-1\right)}\left[\sum_{j=1}^{2}\binom{d_{f}-1}{d_{f}-j}\sum_{k=-1}^{1}k^{j-1}\beta_{\left(k\right)}^{\left(d+j-2\right)}\right]+f^{\left(d_{f}-2\right)}\left[\sum_{k=-1}^{1}\beta_{\left(k\right)}^{\left(d\right)}\right].
\end{align*}
Once again, we know that $\sum_{k=-1}^{1}\beta_{\left(k\right)}^{\left(d\right)}=0$,
which removes the last summand.\\
If $\sum_{k=-1}^{1}k\beta_{\left(k\right)}^{\left(d\right)}=\sum_{k=-1}^{1}\beta_{\left(k\right)}^{\left(d-1\right)}=0$,
then the second summand is zero, and dividing by the nonzero coefficient
$f^{\left(d_{f}\right)}$, we are left with
\[
0=\sum_{k=-1}^{1}\beta_{\left(k\right)}^{\left(d-2\right)}+d_{f}\sum_{k=-1}^{1}k\beta_{\left(k\right)}^{\left(d-1\right)}+\binom{d_{f}}{2}\sum_{k=-1}^{1}k^{2}\beta_{\left(k\right)}^{\left(d\right)}.
\]
We already have that $\sum_{k=-1}^{1}\beta_{\left(k\right)}^{\left(d\right)}=0$
and assumed that $\sum_{k=-1}^{1}k\beta_{\left(k\right)}^{\left(d\right)}$.
If $\sum_{k=-1}^{1}k^{2}\beta_{\left(k\right)}^{\left(d\right)}$
as well, then we must have that $\beta_{\left(k\right)}^{\left(d\right)}=0$
for $k=-1,0,1$, but $d$ was chosen as the max degree of the $\beta_{\left(k\right)}$,
so at least one of the $\beta_{\left(k\right)}^{\left(d\right)}$
cannot be zero. Thus under our assumption we get that $\sum_{k=-1}^{1}k^{2}\beta_{\left(k\right)}^{\left(d\right)}=\beta_{\left(-1\right)}^{\left(d\right)}+\beta_{\left(1\right)}^{\left(d\right)}\neq0$,
so that the quadratic equation above is not trivial.
\end{enumerate}
\end{proof}
Note that once we know the $\beta_{\left(k\right)}$ and the degree
of $f$, \eqref{pol-recursion-formula} in the lemma is a linear system
in the coefficients of $f$, which can easily be solved using standard
methods.

Applying the previous lemma to our case, we get the following:
\begin{thm}
\label{thm:f_algorithm}Suppose that $f,a,b,h_{1},h_{2}\in\CC\left[x\right]$
are polynomials satisfying 
\begin{align}
b\left(x\right) & =-h_{1}\left(x\right)h_{2}\left(x\right)\label{eq:PCF-recursion}\\
f\left(x\right)a\left(x\right) & =f\left(x-1\right)h_{1}\left(x\right)+f\left(x+1\right)h_{2}\left(x+1\right).
\end{align}
Let $d=\max\left\{ \deg\left(a\right),\deg\left(h_{1}\right),\deg\left(h_{2}\right)\right\} $,
$d_{f}=\deg\left(f\right)$ and write 
\begin{align*}
a\left(x\right) & =\sum_{i=0}^{d}a^{\left(i\right)}x^{i} & h_{1}\left(x\right) & =\sum_{i=0}^{d}h_{1}^{\left(i\right)}x^{i} & h_{2}\left(x\right) & =\sum_{i=0}^{d}h_{2}^{\left(i\right)}x^{i}.
\end{align*}
Finally, let $c_{i}$ be the leading coefficients of $h_{i}$ for
$i=1,2$ and $A,B$ be the leading coefficients of $a\left(x\right),b\left(x\right)$
respectively. 

Then the possible triples $\left(c_{1},c_{2},d_{f}\right)$
\begin{enumerate}
\item If $\deg\left(a\right)=\deg\left(h_{1}\right)>\deg\left(h_{2}\right)$,
then $c_{1}=A$, $c_{2}=-\frac{B}{A}$ and $d_{f}=\frac{a^{\left(d-1\right)}-h_{1}^{\left(d-1\right)}-h_{2}^{\left(d-1\right)}}{-A}$.
\item If $\deg\left(a\right)=\deg\left(h_{2}\right)>\deg\left(h_{1}\right)$,
then $c_{1}=-\frac{B}{A}$, $c_{2}=A$ and $d_{f}=\frac{a^{\left(d-1\right)}-h_{1}^{\left(d-1\right)}-h_{2}^{\left(d-1\right)}-dA}{A}$.
\item If $\deg\left(h_{1}\right)=\deg\left(h_{2}\right)>\deg\left(a\right)$,
then $\left\{ c_{1},c_{2}\right\} =\left\{ \sqrt{B},-\sqrt{B}\right\} $
and $d_{f}=\frac{a^{\left(d-1\right)}-h_{1}^{\left(d-1\right)}-h_{2}^{\left(d-1\right)}-dh_{2}^{\left(d\right)}}{h_{2}^{\left(d\right)}-h_{1}^{\left(d\right)}}$.
\item If $\deg\left(h_{1}\right)=\deg\left(h_{2}\right)=\deg\left(a\right)$,
then $\left\{ c_{1},c_{2}\right\} $ are the roots of $x^{2}-Ax-B=0$.
\begin{enumerate}
\item If $c=c_{1}=c_{2}$ (so that $B=-c^{2}$, $A=2c$), then $d_{f}$
is a solution to 
\[
{\scriptstyle \left(a^{\left(d-2\right)}-h_{1}^{\left(d-2\right)}-\left(h_{2}^{\left(d-2\right)}+\left(d-1\right)h_{2}^{\left(d-1\right)}+\binom{d}{2}h_{2}^{\left(d\right)}\right)\right)+d_{f}\left(h_{1}^{\left(d-1\right)}-\left(h_{2}^{\left(d-1\right)}+dh_{2}^{\left(d\right)}\right)\right)-\binom{d_{f}}{2}A=0}.
\]
\item Otherwise, $d_{f}=\frac{a^{\left(d-1\right)}-h_{1}^{\left(d-1\right)}-h_{2}^{\left(d-1\right)}-dh_{2}^{\left(d\right)}}{h_{2}^{\left(d\right)}-h_{1}^{\left(d\right)}}$.
\end{enumerate}
\end{enumerate}
\end{thm}

\begin{proof}
Applying \lemref{polynomial-recursion} to this case we get that 
\begin{enumerate}
\item We have $a^{\left(d\right)}=h_{1}^{\left(d\right)}+h_{2}^{\left(d\right)}$
and at least two of the $h_{1}^{\left(d\right)},h_{2}^{\left(d\right)},a^{\left(d\right)}$
are non zero.
\item If $h_{1}^{\left(d\right)}\neq h_{2}^{\left(d\right)}$, then the
degree of $f$ must be $d_{f}=\frac{a^{\left(d-1\right)}-h_{1}^{\left(d-1\right)}-h_{2}^{\left(d-1\right)}-dh_{2}^{\left(d\right)}}{h_{2}^{\left(d\right)}-h_{1}^{\left(d\right)}}$.
In particular, this expression must be well defined and an integer.
\item If $h_{1}^{\left(d\right)}=h_{2}^{\left(d\right)}$, then $a^{\left(d-1\right)}=h_{1}^{\left(d-1\right)}+h_{2}^{\left(d-1\right)}+dh_{2}^{\left(d\right)}$
and 
\[
{\scriptstyle \left(a^{\left(d-2\right)}-h_{1}^{\left(d-2\right)}-\left(h_{2}^{\left(d-2\right)}+\left(d-1\right)h_{2}^{\left(d-1\right)}+\binom{d}{2}h_{2}^{\left(d\right)}\right)\right)+d_{f}\left(h_{1}^{\left(d-1\right)}-\left(h_{2}^{\left(d-1\right)}+dh_{2}^{\left(d\right)}\right)\right)-\binom{d_{f}}{2}\left(h_{1}^{\left(d\right)}+h_{2}^{\left(d\right)}\right)=0}
\]
is a nontrivial quadratic equation in $d_{f}$ (namely $\left(h_{1}^{\left(d\right)}+h_{2}^{\left(d\right)}\right)\neq0$).
\end{enumerate}
Part (1) shows that at least 2 of the 3 polynomials $a,h_{1},h_{2}$
have the max degree $d$. Separating according to the degree of the
3rd polynomial, and using $a^{\left(d\right)}=h_{1}^{\left(d\right)}+h_{2}^{\left(d\right)}$
and the fact that $c_{1}c_{2}=-B$, we obtain the cases presented
in the theorem.
\end{proof}
\begin{example}
\begin{enumerate}
\item Suppose that we start with a polynomial continued fraction in the
trivial Euler family, namely 
\begin{align*}
b\left(x\right) & =-h_{1}\left(x\right)h_{2}\left(x\right)\\
a\left(x\right) & =h_{1}\left(x\right)+h_{2}\left(x+1\right).
\end{align*}
Guessing the decomposition of $b\left(x\right)$, the algorithm above
should show that $d_{f}=0$, namely we can take $f\equiv1$ constant.
Let's see three examples:
\begin{enumerate}
\item If $b\left(x\right)=-x^{d}\times x^{d}$ and $a\left(x\right)=x^{d}+\left(1+x\right)^{d}$,
then $\deg\left(h_{1}\right)=\deg\left(h_{2}\right)=\deg\left(a\right)=d$.
This leaves us in part (4) of \thmref{f_algorithm}, and with the
polynomial $x^{2}-2x+1$ which has a double root $x=1$, so we are
in case (a). We now have that
\begin{align*}
\begin{array}{c|ccc}
j= & \;d\; & d-1 & d-2\\
\hline a^{\left(j\right)} & 2 & d & \binom{d}{2}\\
h_{1}^{\left(j\right)} & 1 & 0 & 0\\
h_{2}^{\left(j\right)} & 1 & 0 & 0
\end{array},
\end{align*}
so to find $d_{f}$, we need to solve the equation
\begin{align*}
0 & =\left(\binom{d}{2}-\binom{d}{2}\right)-d\cdot d_{f}-2\binom{d_{f}}{2}=-\left(d+d_{f}-1\right)d_{f}.
\end{align*}
Hence either $d_{f}=1-d\leq0$ or $d_{f}=0$, so in any way we know
to look for a constant $f$ solution.
\item If $b\left(x\right)=-\left(-x^{d}\right)\times x^{d}$ and $a\left(x\right)=\left(1+x\right)^{d}-x^{d}$,
so that $\deg\left(h_{1}\right)=\deg\left(h_{2}\right)=d>d-1=\deg\left(a\right)$.
Then we are in case (3) from \thmref{f_algorithm}. We have
\begin{align*}
\begin{array}{c|ccc}
j= & \;d\; & d-1 & d-2\\
\hline a^{\left(j\right)} & 0 & d & \binom{d}{2}\\
h_{1}^{\left(j\right)} & -1 & 0 & 0\\
h_{2}^{\left(j\right)} & 1 & 0 & 0
\end{array}.
\end{align*}
and the degree of $f$ should be
\[
d_{f}=\frac{a^{\left(d-1\right)}-h_{1}^{\left(d-1\right)}-h_{2}^{\left(d-1\right)}-dh_{2}^{\left(d\right)}}{h_{2}^{\left(d\right)}-h_{1}^{\left(d\right)}}=\frac{d-d}{1-\left(-1\right)}=0.
\]
\item Consider $b\left(x\right)=-1\times x$ (so that $\deg\left(h_{1}\right)\neq\deg\left(h_{2}\right)$)
and $a\left(x\right)=1+\left(x+1\right)$, and therefore $d=1$, which
fall in case (2) above. Here the coefficients of $x^{-1}$ are considered
as zero and we get
\begin{align*}
\begin{array}{c|ccc}
j= & \;d\; & d-1 & d-2\\
\hline a^{\left(j\right)} & 1 & 2 & 0\\
h_{1}^{\left(j\right)} & 0 & 1 & 0\\
h_{2}^{\left(j\right)} & 1 & 0 & 0
\end{array}.
\end{align*}
and the degree is
\[
d_{f}=\frac{a^{\left(d-1\right)}-h_{1}^{\left(d-1\right)}-h_{2}^{\left(d-1\right)}-dh_{2}^{\left(d\right)}}{h_{2}^{\left(d\right)}-h_{1}^{\left(d\right)}}=\frac{2-1-0-1}{1-0}=0.
\]
\end{enumerate}
\item Consider $b\left(x\right)=-x^{3}\times x^{3}$ and $a\left(x\right)=x^{3}+\left(1+x\right)^{3}+4\cdot\left(2x+1\right)$,
where $\tilde{h}_{1}\left(x\right)=\tilde{h}_{2}\left(x\right)=x^{3}$,
so that $3=\deg\left(\tilde{h}_{1}\right)=\deg\left(\tilde{h}_{2}\right)=\deg\left(a\right)$.
This puts us in case $4$, and since $A=2,\;B=-1$, we get that $x^{2}-Ax-B=\left(x-1\right)^{2}$
has double root $x=1$, so we are in $4\left(a\right)$, so in any
possible solution we must have $h_{i}=\tilde{h}_{i}$. We now have
that 
\begin{align*}
\begin{array}{c|ccc}
j= & \;d\; & d-1 & d-2\\
\hline a^{\left(j\right)} & 2 & 3 & 11=\binom{3}{2}+8\\
h_{1}^{\left(j\right)} & 1 & 0 & 0\\
h_{2}^{\left(j\right)} & 1 & 0 & 0
\end{array},
\end{align*}
and we are looking for the degree $d_{f}$ as a solution for 
\[
0=8-3d_{f}-d_{f}\left(d_{f}-1\right)=8-2d_{f}-d_{f}^{2}=\left(4+d_{f}\right)\left(2-d_{f}\right).
\]
Since $d_{f}$ needs to be nonnegative, we only need to check $d_{f}=2$.
Solving the linear system will produce $f(x)=x^{2}+x+\frac{1}{2}$.
\item Take $b\left(x\right)=-x^{6}$ and $a\left(x\right)=34x^{3}+51x^{2}+27x+5$
which is the polynomial continued fraction used by Ap\'ery in his
proof of irrationality of $\zeta\left(3\right)$ \cite{apery_irrationalite_1979,van_der_poorten_proof_1979}.
Let us show that in this case there is no solution to \eqref{PCF-recursion}
in the \thmref{f_algorithm}.\\
Assume by negation that there is a solution. In any decomposition
$b\left(x\right)=-h_{1}\left(x\right)h_{2}\left(x\right)$ we have
that $\deg\left(h_{1}\right)+\deg\left(h_{2}\right)=6$. Since $\deg\left(a\right)=3$,
in order for part $\left(1\right)$ in \thmref{f_algorithm} to hold,
we must have that $\deg\left(h_{1}\right)=\deg\left(h_{2}\right)=3$,
and therefore $\tilde{h_{1}}\left(x\right)=\tilde{h}_{2}\left(x\right)=x^{3}$.
Using the \thmref{f_algorithm}, the leading coefficients for $h_{1},h_{2}$
are the roots for $x^{2}-34x+1$. In particular, we can write them
as $c,\frac{1}{c}$ where $c+\frac{1}{c}=34$, and in particular $c\neq\pm1$.\\
It then follows that the degree of $f$ is 
\[
d_{f}=\frac{a^{\left(d-1\right)}-h_{1}^{\left(d-1\right)}-h_{2}^{\left(d-1\right)}-dh_{2}^{\left(d\right)}}{h_{2}^{\left(d\right)}-h_{1}^{\left(d\right)}}=\frac{51-0-0-3}{\frac{1}{c}-c}=\frac{48\cdot c}{\left(1-c^{2}\right)},
\]
where $d_{f}\geq0$ is an integer. It follows that $c$ also satisfies
the quadratic equation $d_{f}c^{2}+48c-d_{f}=0$. Combining the two
quadratic equations we get 
\[
0=\left(d_{f}c^{2}+48c-d_{f}\right)-d_{f}\left(c^{2}-34c+1\right)=\left(48+34\cdot d_{f}\right)c-2d_{f},
\]
so $c$ must be a rational number. However, if $c^{2}-34c+1$ has
a rational root, then its denominator and numerator must divide $1$,
namely the root must be $\pm1$ - contradiction. \\
We conclude that the polynomial continued fraction presentation $\KK_{1}^{\infty}\frac{-n^{6}}{34n^{3}+51n^{2}+27n+5}$
cannot be written as in \eqref{PCF-recursion}.
\end{enumerate}
\end{example}

\section{\label{app:Beta-and-Gamma}Degree 1 Euler continued fraction}

We have already seen several examples of polynomial continued fraction
in \ref{exa:Euler-family} for some well known constants. In this
section we restrict our attention to Euler continued fraction of the
form $\KK_{1}^{\infty}\frac{-h_{1}\left(i\right)h_{2}\left(i\right)}{h_{1}\left(i\right)+h_{2}\left(i\right)}$
where both $h_{1},h_{2}$ are polynomial of degree $1$ over $\ZZ$.

Writing $h_{1}\left(x\right)=ax+b$ and $h_{2}\left(x\right)=cx+d$
with $a,b,c,d\in\ZZ$, we have seen that 
\[
\left(\frac{1}{h_{2}\left(1\right)}\KK_{1}^{n}\frac{b\left(i\right)}{a\left(i\right)}+1\right)^{-1}=\sum_{m=0}^{n}\prod_{i=1}^{m}\left(\frac{h_{1}\left(i\right)}{h_{2}\left(i+1\right)}\right)=\sum_{m=0}^{n}\prod_{i=1}^{m}\left(\frac{ai+b}{ci+c+d}\right).
\]
In order for the right hand side to converge to a finite number,
we may assume that $\left|a\right|\leq\left|c\right|$.

In case where $a,c=1$, so that the roots of $h_{1},h_{2}$ are integers,
then we can write the product above as 
\[
\prod_{i=1}^{m}\left(\frac{i+b}{i+1+d}\right)=\frac{\left(m+b\right)!}{b!}\cdot\frac{\left(1+d\right)!}{\left(m+1+d\right)!}.
\]
This is a rational function $m$, where we can use standard decomposition
of rational functions in order to sum up over it. Note also that if
$b\leq-1$(or $d\leq-2$), then the product is zero (or has division
by zero respectively). A similar behavior happens when $a,c\in\left\{ \pm1\right\} $.

Things become more complicated when $a,c\notin\left\{ \pm1\right\} $.
\begin{example}
Taking $h_{1}\left(x\right)=x$, $h_{2}\left(x\right)=2x-1$ we get
\[
\left(\KK_{1}^{n}\frac{-n\left(2n-1\right)}{3n+1}+1\right)^{-1}=\sum_{k=0}^{n}\prod_{i=1}^{k}\left(\frac{i}{2i+1}\right)=\sum_{k=0}^{n}\frac{k!}{3\cdot5\cdot7\cdots\left(2k-1\right)\cdot\left(2k+1\right)}=\sum_{k=0}^{n}2^{k}\frac{k!k!}{\left(2k+1\right)!}.
\]
We can no longer use rational function decomposition, and standard
summation tricks to find the limit. Luckily for us, the factorials
appearing in the summation can be written as:
\[
\frac{k!k!}{\left(2k+1\right)!}=\int_{0}^{1}x^{k}\left(1-x\right)^{k}\dx.
\]

This is a standard result about the Beta function $B\left(z_{1},z_{2}\right)$
which we will describe below. Once we have this equality, standard
analysis methods show that 
\[
\sum_{k=0}^{n}2^{k}\frac{k!k!}{\left(2k+1\right)!}=\sum_{0}^{n}\int_{0}^{1}2^{k}x^{k}\left(1-x\right)^{k}\dx=\int_{0}^{1}\frac{1-\left(2x\left(1-x\right)\right)^{n+1}}{1-2x\left(1-x\right)}\dx.
\]
Note that $\left|2x\left(1-x\right)\right|\leq\frac{1}{2}$ so when
$n\to\infty$, the numerator converge to 1 uniformly. It follows that
the limit is 
\[
\int_{0}^{1}\frac{1}{1-2x\left(1-x\right)}\dx=2\int_{0}^{1}\frac{1}{\left(2x-1\right)^{2}+1}\dx=\arctan\left(2x-1\right)\mid_{0}^{1}=\frac{\pi}{2}.
\]
\end{example}

As seen in the example above, the idea was to move from the discrete
factorial function to the continuous Beta function, then use tools
from analysis to find the limit. In order to generalize this process,
we begin by recalling the definition and the main results about the
Beta and Gamma functions we need. 
\begin{defn}[\textbf{Beta and Gamma}]
The Beta and Gamma functions of complex variables are defined by
\begin{align*}
B\left(z_{1},z_{2}\right) & =\int_{0}^{1}t^{z_{1}-1}\left(1-t\right)^{z_{2}-1}\dt\\
\Gamma\left(z\right) & =\int_{0}^{\infty}e^{-t}t^{z-1}\dt.
\end{align*}
\end{defn}

\begin{lem}
\label{lem:Beta-Gamma}The Beta and Gamma functions satisfy the following:
\begin{enumerate}
\item \textbf{Gamma Function:}
\begin{enumerate}
\item $\Gamma\left(1\right)=1$ and $\Gamma\left(z+1\right)=z\cdot\Gamma\left(z\right)$
for any $z\in\CC$ with $Re\left(z\right)>1$ . In particular $\Gamma\left(n+1\right)=n!$
for an integer $n\geq0$.
\item For $0<a\leq b$ integers, we have
\[
\prod_{i=0}^{m-1}\left(ai+b\right)=b\left(a+b\right)\left(2a+b\right)\cdots\left(\left(m-1\right)a+b\right)=a^{m}\frac{\Gamma\left(m+\frac{b}{a}\right)}{\Gamma\left(\frac{b}{a}\right)}.
\]
\end{enumerate}
\item \textbf{Gamma To Beta: }For any $z_{1},z_{2}\in\CC$ we have 
\[
B\left(z_{1},z_{2}\right)=\frac{\Gamma\left(z_{1}\right)\Gamma\left(z_{2}\right)}{\Gamma\left(z_{1}+z_{2}\right)}.
\]
\item \textbf{Beta Function:}
\begin{enumerate}
\item For $x>0$ we have
\[
B\left(1,x\right)=\frac{1}{x}.
\]
\item For $x,y>0$ the Beta function satisfies the recurrences:
\[
B\left(x+1,y\right)=B\left(x,y\right)\cdot\frac{x}{x+y}\quad;\quad B\left(x,y+1\right)=B\left(x,y\right)\cdot\frac{y}{x+y}.
\]
\end{enumerate}
\end{enumerate}
\end{lem}

\begin{proof}
These are all well known and elementary results, and are left as exercise
to the reader.
\end{proof}
\newpage{}

With these results, we now have the following:
\begin{claim}
Let $h_{1}\left(x\right)=ax+b$ and $h_{2}\left(x\right)=cx+d$. Suppose
that (1) $0<a\leq c$ , (2) $0\leq b,d$ and (3) $1+\frac{d}{c}>\frac{b}{a}$.
Then:
\begin{align*}
\left(\frac{1}{h_{2}\left(1\right)}\KK_{1}^{n}\frac{-h_{1}\left(i\right)h_{2}\left(i\right)}{h_{1}\left(i\right)+h_{2}\left(i+1\right)}+1\right)^{-1} & =\frac{1}{B\left(1+\frac{b}{a},1+\frac{d}{c}-\frac{b}{a}\right)}\int_{0}^{1}\frac{t^{\frac{b}{a}}\left(1-t\right)^{\frac{d}{c}-\frac{b}{a}}}{1-\frac{a}{c}t}\dt.
\end{align*}
In particular, for $a=c>0$ and $d>b$, the expression above equals
to
\[
\frac{B\left(1+\frac{b}{a},\frac{d-b}{a}\right)}{B\left(1+\frac{b}{a},1+\frac{d-b}{a}\right)}=\frac{d+a}{d-b}.
\]
\end{claim}

\begin{proof}
Using the results on the Beta and Gamma functions from \lemref{Beta-Gamma}
we get that:
\begin{align*}
\left(\frac{1}{h_{2}\left(1\right)}\KK_{1}^{n}\frac{-h_{1}\left(i\right)h_{2}\left(i\right)}{h_{1}\left(i\right)+h_{2}\left(i+1\right)}+1\right)^{-1} & =\sum_{m=0}^{n}\prod_{i=1}^{m}\left(\frac{h_{1}\left(i\right)}{h_{2}\left(i+1\right)}\right)=\sum_{m=0}^{n}\prod_{i=1}^{m}\left(\frac{ai+b}{ci+c+d}\right)\\
 & =\sum_{k=0}^{n}a^{m}\frac{\Gamma\left(1+m+\frac{b}{a}\right)}{\Gamma\left(1+\frac{b}{a}\right)}\frac{1}{c^{m}}\frac{\Gamma\left(2+\frac{d}{c}\right)}{\Gamma\left(2+m+\frac{d}{c}\right)}\\
 & =\frac{\Gamma\left(2+\frac{d}{c}\right)}{\Gamma\left(1+\frac{b}{a}\right)\Gamma\left(1+\frac{d}{c}-\frac{b}{a}\right)}\sum_{m=0}^{n}\left(\frac{a}{c}\right)^{m}\frac{\Gamma\left(1+m+\frac{b}{a}\right)\Gamma\left(1+\frac{d}{c}-\frac{b}{a}\right)}{\Gamma\left(2+m+\frac{d}{c}\right)}\\
 & =\frac{1}{B\left(1+\frac{b}{a},1+\frac{d}{c}-\frac{b}{a}\right)}\overbrace{\sum_{m=0}^{n}\left(\frac{a}{c}\right)^{m}B\left(1+m+\frac{b}{a},1+\frac{d}{c}-\frac{b}{a}\right)}^{=S\left(a,b,c,d\right)}.
\end{align*}

Using the definition of $B\left(z_{1},z_{2}\right)$ we obtain
\[
S\left(a,b,c,d\right)=\sum_{m=0}^{\infty}\int_{0}^{1}\left(\frac{a}{c}\cdot t\right)^{m}\cdot t^{\frac{b}{a}}\left(1-t\right)^{\frac{d}{c}-\frac{b}{a}}\dt.
\]
Under the assumption that $\left|a\right|\leq\left|c\right|$, we
get that $\sum_{0}^{\infty}\left(\frac{a}{c}\right)^{m}t^{m}$ converges
uniformly to $\frac{1}{1-\frac{a}{c}t}$ in $\left[0,1-\varepsilon\right]$
for any $0<\varepsilon<1$. Thus, we can switch the order of summation
and integration to obtain:
\[
S\left(a,b,c,d\right)=\sum_{m=0}^{\infty}\int_{0}^{1}\left(\frac{a}{c}\right)^{m}t^{m}\cdot t^{\frac{b}{a}}\left(1-t\right)^{\frac{d}{c}-\frac{b}{a}}\dt=\int_{0}^{1}\frac{t^{\frac{b}{a}}\left(1-t\right)^{\frac{d}{c}-\frac{b}{a}}}{1-\frac{a}{c}t}\dt.
\]
This completes the first claim.

In case where $a=c>0$ and $d\geq b$ we can write 
\[
S\left(a,b,a,d\right)=\int_{0}^{1}\frac{t^{\frac{b}{a}}\left(1-t\right)^{\frac{d}{a}-\frac{b}{a}}}{1-t}\dt=B\left(1+\frac{b}{a},\frac{d-b}{a}\right).
\]
Using the recurrence $B\left(x,y+1\right)=B\left(x,y\right)\cdot\frac{y}{x+y}$,
we conclude that 
\[
\left(\frac{1}{h_{2}\left(1\right)}\KK_{1}^{n}\frac{-h_{1}\left(i\right)h_{2}\left(i\right)}{h_{1}\left(i\right)+h_{2}\left(i+1\right)}+1\right)^{-1}=\frac{B\left(1+\frac{b}{a},\frac{d-b}{a}\right)}{B\left(1+\frac{b}{a},1+\frac{d-b}{a}\right)}=\frac{B\left(1+\frac{b}{a},\frac{d-b}{a}\right)}{B\left(1+\frac{b}{a},\frac{d-b}{a}\right)\frac{\frac{d-b}{a}}{1+\frac{b}{a}+\frac{d-b}{a}}}=\frac{d+a}{d-b}.
\]
\end{proof}

\newpage{}

\bibliographystyle{plain}
\bibliography{apery}

\end{document}